\documentclass[10pt]{amsart}
\usepackage{amssymb,verbatim,amscd,mathrsfs,stmaryrd,wasysym,latexsym,bbm}
\usepackage[all]{xy} 
\begin{document}

\newcommand{\V}{{\mathcal V}}     
\renewcommand{\O}{{\mathcal O}}
\newcommand{\LL}{\mathcal L}
\newcommand{\mcO}{\mathcal{O}}
\newcommand{\mcF}{\mathcal{F}}
\newcommand{\mcL}{\mathcal{L}}
\newcommand{\mcG}{\mathcal{G}}
\newcommand{\mcH}{\mathcal{H}}
\newcommand{\mcI}{\mathcal{I}}
\newcommand{\mcJ}{\mathcal{J}}
\newcommand{\mcK}{\mathcal{K}}

\newcommand{\scr}{\mathscr}

\newcommand{\Ext}{\hbox{\rm Ext}}
\newcommand{\Tor}{\hbox{\rm Tor}}
\newcommand{\Hom}{\hbox{Hom}}
\newcommand{\Spec}{\hbox{Spec }}
\newcommand{\Proj}{\hbox{Proj}\ }
\newcommand{\Mod}{\hbox{Mod}}
\newcommand{\GrMod}{\hbox{GrMod}}
\newcommand{\grmod}{\hbox{gr-mod}}
\newcommand{\Tors}{\hbox{Tors}}
\newcommand{\gr}{\hbox{gr}}
\newcommand{\tors}{\hbox{tors}}
\newcommand{\rank}{\hbox{rank}}
\newcommand{\End}{\hbox{{\rm End}}}
\newcommand{\Der}{\hbox{Der}}
\newcommand{\GKdim}{\hbox{GKdim}}
\newcommand{\gldim}{\hbox{gldim}}
\newcommand{\im}{\hbox{im}}
\renewcommand{\ker}{\hbox{ker}\ }
\newcommand{\Char}{\hbox{char}}
\newcommand{\colim}{\hbox{colim}}
\newcommand{\depth}{\hbox{depth}}
\def\bee{\begin{eqnarray}}
\def\eee{\end{eqnarray}}

\newcommand{\AGr}{\hbox{A-Gr}}
\newcommand{\lonto}{{\protect \longrightarrow\!\!\!\!\!\!\!\!\longrightarrow}}

\renewcommand{\c}{\cancel}
\newcommand{\fh}{\frak h}  
\newcommand{\fp}{\frak p}
\newcommand{\fq}{\frak q}
\newcommand{\fr}{\frak r}
\newcommand{\mf}{\mathfrak}
\newcommand{\m}{{\mu}}
\newcommand{\gl}{{\frak g}{\frak l}}
\newcommand{\ssl}{{\frak s}{\frak l}}
\newcommand{\tw}{{\rm tw}}

\newcommand{\ds}{\displaystyle}
\newcommand{\s}{\sigma}
\renewcommand{\l}{\lambda}
\renewcommand{\a}{\alpha}
\renewcommand{\b}{\beta}
\newcommand{\G}{\Gamma}
\newcommand{\U}{\Upsilon}
\newcommand{\g}{\gamma}
\newcommand{\z}{\zeta}
\newcommand{\e}{\epsilon}
\renewcommand{\d}{\delta}
\newcommand{\p}{\rho}
\renewcommand{\t}{\tau}
\newcommand{\n}{\eta}
\newcommand{\x}{\chi}
\newcommand{\w}{\omega}
\renewcommand{\i}{\iota}
\renewcommand{\v}{\nu}

\newcommand{\A}{{\mathbb A}}
\newcommand{\C}{{\mathbb C}}
\newcommand{\N}{{\mathbb N}}
\newcommand{\Z}{{\mathbb Z}}
\newcommand{\ZZ}{{\mathbb Z}}
\newcommand{\Q}{{\mathbb Q}}
\renewcommand{\k}{\mathbb K}

\newcommand{\E}{{\mathcal E}}
\newcommand{\K}{{\mathcal K}}
\renewcommand{\L}{{\mathcal L}}
\renewcommand{\S}{{\mathcal S}}
\newcommand{\T}{{\mathcal T}}

\newcommand{\scrA}{{\mathscr A}}
\newcommand{\scrB}{{\mathscr B}}
\newcommand{\scrP}{{\mathscr P}}
\newcommand{\scrS}{{\mathscr S}}
\newcommand{\scrT}{{\mathscr T}}

\newcommand{\Sch}{{\tt Sch}}

\newcommand{\mfZ}{{\mathbf Z}}
\newcommand{\mfP}{{P}}
\newcommand{\mbfB}{{\mathbf B}}
\newcommand{\mfG}{{\G}}

\newcommand{\GL}{{GL}}
\newcommand{\Sym}{{\mathbf{Sym}}}
\newcommand{\kVect}{{$\k$-Vect}}

\newcommand{\rowxy}{(x\ y)}
\newcommand{\colxy}{ \left({\begin{array}{c} x \\ y \end{array}}\right)}
\newcommand{\scolxy}{\left({\begin{smallmatrix} x \\ y
\end{smallmatrix}}\right)}

\renewcommand{\P}{{\mathbb P}}

\newcommand{\la}{\langle}
\newcommand{\ra}{\rangle}
\newcommand{\tensor}{\otimes}
\newcommand{\tsr}{\tensor}
\newcommand{\ol}{\overline}
\newcommand{\isom}{\cong}
\newcommand{\xto}{\xrightarrow}
\newcommand{\mc}{\mathcal}

\newtheorem{thm}{Theorem}[section]
\newtheorem{lemma}[thm]{Lemma}
\newtheorem{cor}[thm]{Corollary}
\newtheorem{prop}[thm]{Proposition}
\newtheorem{claim}[thm]{Claim}

\theoremstyle{definition}
\newtheorem{defn}[thm]{Definition}
\newtheorem{notn}[thm]{Notation}
\newtheorem{ex}[thm]{Example}
\newtheorem{rmk}[thm]{Remark}
\newtheorem{rmks}[thm]{Remarks}
\newtheorem{note}[thm]{Note}
\newtheorem{example}[thm]{Example}
\newtheorem{problem}[thm]{Problem}
\newtheorem{ques}[thm]{Question}
\newtheorem{conj}[thm]{Conjecture}
\newtheorem{ass}[thm]{Assumptions}
\newtheorem{const}[thm]{Construction}
\newtheorem{thingy}[thm]{}

\newcommand{\onto}{{\protect \rightarrow\!\!\!\!\!\rightarrow}}
\newcommand{\donto}{\put(0,-2){$|$}\put(-1.3,-12){$\downarrow$}{\put(-1.3,-14.5) 

{$\downarrow$}}}

\newcounter{letter}
\renewcommand{\theletter}{\rom{(}\alph{letter}\rom{)}}

\newenvironment{lcase}{\begin{list}{~~~~\theletter} {\usecounter{letter}
\setlength{\labelwidth4ex}{\leftmargin6ex}}}{\end{list}}

\newcounter{rnum}
\renewcommand{\thernum}{\rom{(}\roman{rnum}\rom{)}}

\newenvironment{lnum}{\begin{list}{~~~~\thernum}{\usecounter{rnum}
\setlength{\labelwidth4ex}{\leftmargin6ex}}}{\end{list}}

\thispagestyle{empty}

\title[Algebras associated to inverse systems of projective schemes]{Algebras associated to inverse systems of projective schemes}

\keywords{Artin-Schelter regular algebras, noncommutative algebraic geometry, Sklyanin algebras}

\author[  Conner, Goetz ]{ }

  \subjclass[2020]{16S38, 16W50}
\maketitle

\begin{center}

\vskip-.2in Andrew Conner \\
\bigskip

Department of Mathematics and Computer Science\\
Saint Mary's College of California\\
Moraga, CA 94575\\
\bigskip

 Peter Goetz \\
\bigskip

Department of Mathematics\\ Cal Poly Humboldt\\
Arcata, California  95521
\\ \ \\

\end{center}

\setcounter{page}{1}

\thispagestyle{empty}

\vspace{0.2in}

\begin{abstract}
Artin, Tate and Van den Bergh initiated the field of noncommutative projective algebraic geometry by fruitfully studying geometric data associated to noncommutative graded algebras. More specifically, given a field $\k$ and a graded $\k$-algebra $A$, they defined an inverse system of projective schemes $\U_A = \{{\Upsilon_d(A)}\}$. This system affords an algebra, $\mbfB(\U_A)$, built out of global sections, and a $\k$-algebra morphism $\t: A \to \mbfB(\U_A)$. We study and extend this construction. We define, for any natural number $n$, a category ${\tt PSys}^n$ of \emph{projective systems of schemes} and a contravariant functor $\mbfB$ from ${\tt PSys}^n$ to the category of associative $\k$-algebras. We realize the schemes ${\Upsilon_d(A)}$ as $\Proj {\mathbf U}_d(A)$, where ${\mathbf U}_d$ is a functor from associative algebras to commutative algebras. We characterize when the morphism $\t: A \to \mbfB(\U_A)$ is injective or surjective in terms of local cohomology modules of the ${\mathbf U}_d(A)$. Motivated by work of Walton, when $\U_A$ consists of well-behaved schemes, we prove a geometric result that computes the Hilbert series of $\mbfB(\U_A)$. We provide many detailed examples that illustrate our results. For example, we prove that for some non-AS-regular algebras constructed as twisted tensor products of polynomial rings, $\t$ is surjective or an isomorphism. 
\end{abstract}

\bigskip


\section{Introduction}
\label{introduction}

In their seminal paper \cite{ATVI} the authors introduced a new way to study graded algebras using algebraic geometry. The central focus of that work is the study of what are now called Artin-Schelter (AS) regular algebras of dimension 3. To each such algebra $A$, one can associate a geometric triple $(E, \s, \mcL)$, consisting of a scheme $E$, an automorphism $\s:E\to E$ and a base-point-free line bundle $\mcL$ on $E$. Likewise, given any such triple, one can define the \emph{twisted homogeneous coordinate ring} $B(E,\s,\mcL)$, which the authors show is either isomorphic to $A$, or is a quotient of $A$ by a normal, regular homogeneous element whose degree is larger than those of the defining relations of $A$. The ability to recover $A$ from $B(E,\s,\mcL)$ enabled the authors to classify, and to infer ring-theoretic properties of, three-dimensional AS-regular algebras. It is therefore natural to hope that geometric information can also be used to study AS-regular algebras of higher dimensions, and even graded algebras that are not AS-regular. 

To that end, in this paper we investigate a different geometric construction from \cite{ATVI} - one that can be applied to any connected, $\N$-graded $\k$-algebra, finitely generated in degree 1. Given a graded algebra $A$ we define projective schemes $\U_d(A)\subseteq (\P^n)^{\times d}$ that parametrize isomorphism classes of \emph{truncated point modules of length $d+1$} for $A$. We then define the structure of a graded ring on the space  $\mbfB(A)=\bigoplus_d H^0(\U_d(A), i^*\O_{(\P^n)^{\times d}}(1))$, and a canonical graded algebra morphism $\t_A: A \to \mbfB(A)$, which is bijective in degree 1. In the case when $A$ is a quadratic, AS-regular algebra of dimension 3, it can be shown that $\mbfB(A)\cong B(E,\s,\mcL)$,  and that $\t_A$ is surjective with $\ker \ \t_A$ generated in degree 3. However, in general, $\t_A$ need not be injective nor surjective; see Example 3 in Section \ref{examples}. 

The primary goal of this paper is to provide some partial answers to two related, natural questions:  
\begin{enumerate}
\item What conditions guarantee that  $A$ can be recovered from $\mbfB(A)$ via $\t_A$?
\item When is $\mbfB(A)$  amenable to study via geometry?
\end{enumerate}

In \cite{Walton1} (see also \cite{Walton2}) the author showed that if $S$ is a certain degenerate Sklyanin algebra on three generators, then $\t_S$ is an isomorphism in degrees $\le 4$. Since $S$ is quadratic in this context, $S$ is isomorphic to the quadratic closure of $\mbfB(S)$\footnote{In fact, Walton conjectures that $S\cong \mbfB(S)$.}. Techniques from algebraic geometry are used to obtain these results. 

To our knowledge, in the more than thirty years since \cite{ATVI} appeared, \cite{Walton1} and \cite{Walton2} are the only papers in the literature that attempt to analyze a ring of the form $\mbfB(A)$ when $A$ is not AS-regular. Thus our secondary objective is to create an access point to this subject for researchers in noncommutative algebra by detailing the geometric aspects quite thoroughly.

To state our main results, let $\k$ be an algebraically closed field. Let $A$ denote a connected, $\N$-graded $\k$-algebra, finitely generated in degree 1.
Towards question (1), we define a commutative graded $\k$-algebra ${\mathbf U}_d(A)$ such that $\U_d(A)\cong \Proj {\mathbf U}_d(A)$. When $A$ is commutative, ${\mathbf U}_d(A)$ is isomorphic to the $d$-th Veronese subalgebra of $A$ (Theorem \ref{functor on commutative}). Once this setup is established, our first result follows quickly.

\begin{thm}[Theorem \ref{kernel-image via local cohomology}]
\label{intro1} \ 
\begin{enumerate}
\item The map $\t_A$ is injective if and only if $H^0_{\mf m}({\mathbf U}_d(A))_1=0$ for all $d\ge 2$.
\item The map $\t_A$ is surjective if and only if $H^1_{\mf m}({\mathbf U}_d(A))_1=0$ for all $d\ge 2$.
\end{enumerate}
In particular, if $\depth\ {\mathbf U}_d(A)\ge 2$ for all $d\ge 2$, then $A\cong  \mbfB(A)$.
\end{thm}

If the map $\t_A$ is not injective, the ability to recover $A$ from $\mbfB(A)$ depends on the degrees in which $\ker\ \t_A$ is nonzero. The result above provides one approach to understanding $\ker\ \t_A$. Another point of view is to examine annihilators of point modules.

\begin{thm}[Corollary \ref{kernel of tau}]
\label{intro2}
 For $d\ge 2$, let  
$$J_d=\{f\in A_d\ :\ Pf=0 \text{ for every truncated right point module $P$ of length $d+1$}\}.$$
Then $\ker\ (\t_A)_d\subseteq J_d$, with equality if $\U_d(A)$ is reduced.
\end{thm}

In fact, we prove that $J_d$ can be identified with the degree-1 part of the graded nilradical ${\rm Nil}({\mathbf U}_d(A))$. The containment in the theorem can be strict when $\U_d(A)$ is not reduced, see Example \ref{Ore case 4}. 

In \cite[Section 4.1]{Walton1}, Walton uses the normalization of $\U_d(S)$ to compute the dimension of $\mbfB(S)_d$ where $S$ is a degenerate Sklyanin algebra. This motivated our contribution to question (2). The following special case of Theorem \ref{main hs theorem} illustrates how the theorem can be used in practice; see also Example 1 in Section \ref{examples}. 

\begin{thm}\label{intro3}
Let $X=\U_d(A)$. Assume $X$ is reduced, with discrete singular locus ${\rm Sing}(X)$. Let $(X',\nu)=(\coprod W'_i,\coprod f_i)$ be the normalization of $X$, where the $f_i:W_i'\to X$ are closed immersions of normal schemes.
If $|f_i(W'_i)\cap {\rm Sing}(X)|\le 2$, then
$$\dim_{\k} \mbfB(A)_d = \sum_i h^0(\O_{W'_i}(1)) - |{\rm Sing}(X)|.$$
\end{thm}

The results of this paper apply to graded algebras in general, but we were initially motivated by \cite{CG3} and \cite{CG1}, wherein we classified quadratic twisted tensor products of the form $A=\k[x,y]\tsr_{\t}\k[z]$ up to  isomorphism and determined which are AS-regular. Any such algebra $A$ has the same Hilbert series as a quadratic 3-dimensional AS-regular algebra, so we wondered if $\t_A$ is always surjective when $A=\k[x,y]\tsr_{\t}\k[z]$ is quadratic. Examples 1 and 2 in Section \ref{examples} illustrate that $\t_A$ is surjective for two of the non-AS-regular algebras from \cite{CG1}.


The paper is organized as follows. In Section 2 we provide definitions and background results needed to define $\U_d(A)$ and $\mbfB(A)$. In Section 3 we describe the algebra ${\mathbf U}_d(A)$ in two different ways, show that ${\mathbf U}_d(A)$ is the $d$-th Veronese subalgebra of $A$ when $A$ is commutative, and
prove that $\U_d(A)\cong \Proj {\mathbf U}_d(A).$ 
Section 4 concerns the morphism $\t_A: A \to \mbfB(A)$. We describe when it is surjective, and characterize its kernel in terms of annihilators of point modules. In Section 5 we prove a purely geometric result computing the dimension of the space of global sections of $\mc{O}_X(1)$ for certain well-behaved schemes $X$; Theorem \ref{intro3} is an immediate consequence. Finally in Section 6 we include three examples illustrating our results. 

\section{Preliminaries}
\label{preliminaries}

In this paper we work with vector spaces, algebras, and schemes over a fixed algebraically closed field $\k$. 
Let ${\tt Sch}$ denote the category of $\k$-schemes. When referring to a \emph{scheme}, we always mean an object of ${\tt Sch}$.  Furthermore, all fiber products of schemes are taken with respect to ${\rm Spec}\ \k$. 

By \emph{vector space} we always mean an object of ${\tt Vect}$, the category of $\k$-vector spaces. We denote the tensor product of vector spaces $V$ and $W$ by $V\tsr W$. When $V$ and $W$ are $\N$-graded, 
the space $V\tsr W$ admits an $\N$-grading via the formula
$$(V\tsr W)_m=\bigoplus_{k+\ell=m} V_k\tsr W_{\ell}.$$ We write $V^*$ for the linear dual of the vector space $V$, and require morphisms of graded vector spaces to preserve degrees. 

Whenever we refer to a \emph{graded algebra}, we always mean a connected, $\N$-graded, locally finite-dimensional $\k$-algebra. Let ${\tt Alg}$ denote the category of graded algebras with degree-preserving algebra morphisms. 

Objects of ${\tt Alg}$ need not be generated in degree 1. Let ${\tt Alg}^1$ denote the full subcategory of ${\tt Alg}$ consisting of all 1-generated algebras. If $A\in {\tt Alg}^1$, then there is a canonical surjective morphism of 1-generated graded algebras $\pi_A:{\mathbf T}(A_1)\to A$, where ${\mathbf T}(A_1)$ denotes the tensor algebra, graded by tensor degree. The tensor algebra construction is functorial, so for any morphism $f:A\to B$ in ${\tt Alg}^1$ there is a map ${\mathbf T}(f_1):{\mathbf T}(A_1)\to {\mathbf T}(B_1)$. Moreover, it is clear that ${\mathbf T}(f_1)(\ker \, \pi_A)\subset \ker \, \pi_B$.

We refer to a graded algebra $A\in {\tt Alg}^1$ as \emph{quadratic} if the ideal $\ker \,\pi_A\subset {\mathbf T}(A_1)$ is generated in degree 2.

If $A$ is a graded algebra, we denote by ${\tt Mod}(A)$ the category of $\Z$-graded right $A$-modules, with degree-preserving morphisms. For $M\in {\tt Mod}(A)$ and $d\in \Z$, we denote by $M(d)$ the degree-shifted graded right $A$-module where $M(d)_n=M_{n+d}$ for all $n\in\Z$.
~\\

\subsection{The scheme of zeros and local cohomology.}~\\

Let $X$ be a scheme, $\mcF$ a locally free $\O_X$-module, and $s_1,\ldots, s_n\in H^0(X,\mcF)$ global sections, viewed as morphisms $s_i:\O_X\to \mcF$. Define $\mcI_{s_i}=\im s_i^{\vee}$ to be the image of the dual morphism $s_i^{\vee}:\mcF^{\vee}\to \O_X^{\vee}\cong \O_X$. Since $\mcF$ is locally free, each $\mcI_{s_i}$ is a quasicoherent sheaf of ideals, hence so is $\mcI=\sum_{i=1}^n \mcI_{s_i}$. The \emph{scheme of zeros} of $s_1,\ldots, s_n$ is the closed subscheme $Z\subset X$ determined by $\mcI$, or equivalently, the scheme-theoretic intersection of the closed subschemes determined by the $\mcI_{s_i}$. As a subspace of $X$, $Z={\rm Supp}(\O_X/\mcI)$, and the structure sheaf on $Z$ is the unique (up to isomorphism) sheaf $\O_Z$ such that $i_*\O_Z=\O_X/\mcI$, where $i: Z \to X$ is the inclusion map. 

Here we collect a few general facts concerning schemes of zeros that will be needed later. 



\begin{lemma} 
\label{pullback of image inclusion is injective}
Let $f: X \to Y$ be a morphism of ringed spaces with associated morphism $f^{\sharp}: f^{-1} \mcO_Y \to \mcO_X$ .  Let $\varphi: \mcG \to \mcH$ be a morphism of $\mcO_Y$-modules. Let $\mcI = \im(\varphi)$ and let $\eta: \mcI \to \mcH$ be the canonical inclusion. Then $f^{*}\eta: f^* \mcI \to f^{*} \mcH$ is injective.
\end{lemma}

\begin{proof}
Let $x \in X$. It suffices to check that the map $(f^* \eta)_x: (f^* \mcI)_x \to (f^* \mcH)_x$ is injective. Making the usual natural identifications of stalks at $x$, it suffices to prove that the map $$\eta_{f(x)} \tsr {\rm id}_{\mcO_{X, x}} : \mcI_{f(x)} \tsr_{\mcO_{Y, f(x)}} \mcO_{X, x} \to \mcH_{f(x)} \tsr_{\mcO_{Y, f(x)}} \mcO_{X, x}$$ is injective.


It follows from the restriction/induction of scalars adjunction that the map $$\eta_{f(x)} \tsr {\rm id}_{\mcO_{X, x}} : \mcI_{f(x)} \tsr_{\mcO_{Y, f(x)}} \mcO_{X, x} \to \im(\varphi_{f(x)} \tsr {\rm id}_{\mcO_{X, x}})$$ is an isomorphism. Since $\im(\varphi_{f(x)} \tsr {\rm id}_{\mcO_{X, x}})$ includes into $\mcH_{f(x)} \tsr_{\mcO_{Y, f(x)}} \mcO_{X, x}$, we conclude that $(f^* \eta)_x$ is injective. 

\end{proof}

\begin{lemma}
\label{induced morphism on scheme of zeros}
Let $f:X\to Y$ be a morphism of schemes and $\mcF$ a locally free $\O_Y$-module of finite rank. Let $s_1,\ldots, s_n\in H^0(Y,\mcF)$ and let $s'_i=f^*(s_i)\in H^0(X,f^*\mcF)$ for $1\le i\le n$. Let $Z$, respectively $Z'$, be the scheme of zeros of $s_1,\ldots, s_n$, respectively of $s'_1,\ldots, s'_n$. Let $\iota:Z\to Y$ and $\iota':Z'\to X$ be the corresponding inclusions. Then $f\circ \iota'$ factors through $\iota$.

\end{lemma}

\begin{proof}
%
%

Adopting the notation of the paragraph preceding Lemma \ref{pullback of image inclusion is injective}, let $\mcI = \sum_{i = 1}^n \mcI_{s_i}$ and $\mcI' = \sum_{i = 1}^n \mcI_{f^{\ast}(s_i)}$. Let $\eta: \mcI \to \mcO_Y$ be the canonical inclusion. 

We have the following well-known isomorphisms:
\begin{align*}
f^{\ast} \mcI &= f^{\ast}(\im(\oplus_{i} s_i^{\vee}: \oplus_i \mcF^{\vee} \to \O_Y)) \\
&\cong \im(f^{\ast}(\oplus_i s_i^{\vee}): f^{\ast}(\oplus_{i} \mcF^{\vee}) \to \O_X), \quad \text{$f^{\ast}$ commutes with images} \\
&\cong \im(\oplus_i f^{\ast}(s_i^{\vee}): \oplus_{i} f^{\ast}(\mcF^{\vee}) \to \O_X), \quad \text{$f^{\ast}$ commutes with direct sums} \\
&\cong \im(\oplus_{i} (f^{\ast}(s_i))^{\vee}: \oplus_i (f^{\ast} \mcF)^{\vee} \to \O_X), \quad \text{$\mcF$ is locally free of finite rank} \\
&= \mcI'.
\end{align*}




By Lemma \ref{pullback of image inclusion is injective}, the morphism $f^* \eta: f^* \mcI \to f^* \mcO_Y = \mcO_X$ is injective, so the ideal sheaf $\mcJ = \im(f^* \eta)$ is isomorphic to $f^*\mcI$. Since $f^* \mcI \cong \mcI'$, we conclude that $\mcJ \cong \mcI'$. The result follows from \cite[Lemma 26.4.7]{stacks}.



\end{proof}

\begin{lemma}
\label{scheme of zeros kills global sections}
Let $\mcF$ be a locally free sheaf of finite rank on $X$, let $s\in H^0(X,\mcF)$ and suppose $Z$ is the scheme of zeros of $s$, with inclusion map $i: Z \to X$. Then the image of $s$ under the canonical map $H^0(X,\mcF)\to H^0(Z,i^*\mcF)$ is zero.
\end{lemma}

\begin{proof}
The natural transformation $-\tsr_{\O_X} i_*\O_Z\to i_*i^*(-)$ is an isomorphism on locally free $\O_X$-modules of finite rank. Since $i:Z\to X$ is a closed immersion, $i_*$ is faithful, it suffices to prove that $s\tsr{\rm id}_{i_*\O_Z}=0$. 

Let $U\subset X$ be an affine open. Let $s|_U(1)=t\in \mcF(U)$. Then the dual map $s^{\vee}|_U:\mcF^{\vee}(U)\to \O_X^{\vee}(U)\cong \O_X(U)$  is completely determined by $s^{\vee}|_U(\varphi)(1)=\varphi(t)$. Now, we have canonical isomorphisms
$$\mcF\tsr_{\O_X} \O_X/\mcI\cong (\mcF^{\vee})^{\vee}\tsr_{\O_X} \O_X/\mcI\cong {\mathscr Hom}_{\O_X}(\mcF^{\vee},\O_X/\mcI)$$
and composing with $s\tsr{\rm id}_{i_*\O_Z}$ we get a map
$$\O_X\tsr_{\O_X} \O_X/\mcI\to {\mathscr Hom}_{\O_X}(\mcF^{\vee},\O_X/\mcI).$$
Since $\mcI$ is quasicoherent, $(\O_X/\mcI)(U)=\O_X(U)/\mcI(U)$, so on affine opens, the map above takes $1\tsr q \mapsto (\varphi \mapsto \varphi(t)\cdot q)$. But $\varphi(t)\in \mcI(U)$ so the composite is 0 on each piece of an affine open base. Hence $s\tsr{\rm id}_{i_*\O_Z}=0$, as desired.

\end{proof}

Now suppose $X=\Proj S$ for a commutative, Noetherian, graded algebra $S$. Let ${\mf m}=S_{\ge 1}$ be the irrelevant ideal. Denote by $H_{\mf m}^i(-)$ the $i$-th right derived functor of the ${\mf m}$-torsion functor $\G_{\mf m}:{\tt Mod}(S)\to {\tt Mod}(S)$. Recall that $\G_{\mf m}$ is defined on objects by $\G_{\mf m}(M)=\{m\in M\ :\ \mf m^km=0\ \text{ for some } k\in \N\}$. Note that $H_{\mf m}^i(S)$ admits a natural $\Z$-grading. There is an exact sequence of graded $S$-modules
\begin{equation}
\label{four-term-sequence}
0\to H_{\mf m}^0(S)\to S \xto{\z(S)} \G_*(\widetilde S)\to H_{\mf m}^1(S)\to 0
\end{equation}
where $\G_*(\widetilde S)=\bigoplus_{n\in \Z} H^0(X, \O_X(n))$. The module $\G_*(\widetilde S)$ admits a canonical ring structure, and in fact $\z(S)$ is a homomorphism of graded rings, see \cite[Theorem 13.21]{24hours} for example. We note that $\z(S)$ can be defined by choosing an affine open cover $X=\bigcup_{j=1}^{\ell} D_+(f_j)$ where $f_1,\ldots, f_{\ell}\in S_+$ are homogeneous.
Then for a homogeneous element $x\in S_d$, the image $\z(S)(x)\in H^0(X,\O(d))$ is the unique global section whose restriction to each $D_+(f_j)$ corresponds to the morphism $S_{(f_j)}\to S(d)_{(f_j)}$ of free $S_{(f_j)}$-modules given by $1\mapsto (x/f_j)\cdot f_j$.

For a morphism of schemes $f:X\to Y$, and a sheaf $\mathcal G$ of $\O_Y$-modules, let $\hat{f}:H^0(Y,\mathcal G)\to H^0(X, f^*\mathcal G)$ be the canonical map. When $\mathcal G=\O_Y(n)$ we denote this map by $\hat{f}(n)$.
 We recall that the sequence (\ref{four-term-sequence}) is functorial in the following sense.

\begin{lemma}
\label{naturality of four-term sequence}
Let $S$ and $T$ be commutative, Noetherian, graded algebras, and let $f:S\to T$ be a graded algebra surjection. Let $\mf m=S_{\ge 1}$, $\mathfrak n=T_{\ge 1}$, $X=\Proj S$ and $Y=\Proj T$. Then there is a canonical closed immersion $i:Y\to X$ and isomorphisms $i^*(\O_X(n))\cong \O_Y(n)$ for all $n\in \Z$ such that the following diagram, whose rows are truncations of (\ref{four-term-sequence}), commutes.\\
\centerline{\xymatrix{
0 \ar[r] & H^0_{\mf m}(S) \ar[dd]^{f} \ar[r] & S \ar[dd]^{f} \ar[r]^-{\z(S)} &  \G_*(\widetilde S) \ar[d]^{\oplus \hat{\imath}(n)}  \\
& & & \displaystyle\bigoplus_{n \ge 0} H^0(Y, i^*\O_X(n))\ar[d]^{\cong}\\
0 \ar[r] & H^0_{\mathfrak n}(T) \ar[r] & T \ar[r]^-{\z(T)} &  \G_*(\widetilde T)
}}

\end{lemma}

\begin{proof}
The commutativity of the rightmost square and the fact that $i^*(\O_X(n))\cong \O_Y(n)$ are standard, see \cite[Lemma 27.11.5]{stacks}, for example. Commutativity of the left square then follows from exactness of the rows.

\end{proof}

When global sections $s_1,\ldots, s_n\in H^0(X, \O_X(1))$ are in the image of $\z(S)_1$, the following lemma describes the scheme of zeros $Z$, defined above in geometric terms, in terms of the $\Proj$ construction. Proving the lemma is routine. Nonetheless, it is an important link between our results and the relevant literature, so we include a sketch.

\begin{lemma}
\label{scheme of zeros as proj}
Assume $X=\Proj S$ for a commutative, graded algebra $S$, and let $\mcF=\O_X(1)$. Let $s_1,\ldots, s_n\in \im \, \z(S)_1$ be global sections of $\mcF$ and let $Z$ be the scheme of zeros of the $s_i$. Then $Z\cong \Proj S/\la \tilde s_1,\ldots, \tilde s_n\ra$ for any $\tilde s_1,\ldots, \tilde s_n\in S_1$ satisfying $\z(S)(\tilde s_i)=s_i$ for all $i=1,\ldots, n.$
\end{lemma}

\begin{proof}
Let  $s_1,\ldots, s_n\in \im \z(S)_1$ and choose $\tilde s_1,\ldots, \tilde s_n\in S_1$ satisfying $\z(\tilde s_i)=s_i$ for all $i=1,\ldots, n.$ By definition, $\mcI_Z$ is the sheaf of ideals corresponding to $\sum_{i=1}^n \mcI_{s_i}$ under $\O_X^{\vee}\cong \O_X$. We describe $\mcI_Z$ on an affine open cover of $X$. 

Let $f_1,\ldots, f_g\in S_1$ be a minimal generating set for $S$. Fix $1\le i\le n$ and $1\le j\le g$ and let $U=D_+(f_j)$. 
Since $U$ is affine, we have $\mathcal I_{s_i}(U)=\im s_i^{\vee}(U)$ as $\O_X(U)$-modules. The morphism $s_i^{\vee}(U):\mcF^{\vee}(U)\to \O_X^{\vee}(U)$ is given by $\phi\mapsto \phi\circ s_i|_U$, and its image corresponds to an ideal of $\O_X(U)$ via the identification
$$\O_X^{\vee}(U)\cong \Hom_{S_{(f_j)}}(S_{(f_j)}, S_{(f_j)})\cong S_{(f_j)}=\O_X(U).$$
Since $f_j$ freely generates the $S_{(f_j)}$-module $S(1)_{(f_j)}$, an element 
$$\varphi\in \Hom_{S_{(f_j)}}(S(1)_{(f_j)}, S_{(f_j)})\cong \mcF^{\vee}(U)$$ is determined by $\varphi(f_j)$. Also recall that $s_i|_U$ corresponds to the map 
$$S_{(f_j)}\to S(1)_{(f_j)}, \qquad 1\mapsto (\tilde s_i/f_j)\cdot f_j.$$
Under these identifications, $s_i^{\vee}(U)$ maps $\varphi\mapsto (\tilde s_i/f_j)\varphi(f_j)$. Since $\varphi(f_j)$ is arbitrary, the image of $s_i^{\vee}(U)$ corresponds to the ideal $\la \tilde s_i/f_j\ra\subset \O_X(U)$. Summing over all sections, we get 
\begin{equation}
\label{sheaf-of-ideals}
\mcI_Z(U)=\la \tilde s_1/f_j,\ldots, \tilde s_n/f_j\ra\subset \O_X(U).
\end{equation}
The unique saturated ideal $I_Z$ of $S$ such that $Z\cong \Proj S/I_Z$ is
$$I_Z=\bigoplus_{d\in \N} I_d=\bigoplus_{d\in \N}\{ a\in S_d\ |\ a/f_j^d \in \mcI_Z(D_+(f_j)) \text{ for all }1\le j\le g\}.$$
Using (\ref{sheaf-of-ideals}), $I_Z$ is easily seen to be the saturation of $\la \tilde s_1,\ldots, \tilde s_n\ra$.

\end{proof}

\subsection{Truncated point schemes}\label{Tps}~\\

For any $d\ge 1$ and any 1-generated, graded algebra $A$, the authors of \cite{ATVI} define a scheme $\U_d(A)$ whose closed points are in bijective correspondence with isomorphism classes of \emph{truncated point modules of length $d+1$}; that is, graded $A$-modules $M$ that are generated by $M_0$, and $\dim_{\k} M_i=1$ for $0\le i\le d$ and $\dim_{\k} M_i=0$ otherwise. 

 Let $A\in {\tt Alg}^1$, so $A={\mathbf T}(V)/I$ for $V=A_1$ a finite-dimensional vector space and a homogeneous ideal $I\subset {\mathbf T}(V)$. Let $\P=\P(V^*)$. If $r\in V^{\tsr d}$, we may view $r$ as a global section $\widetilde{r}$ of the sheaf $\O_{\P^{\times d}}(1):={\rm pr}_1^*\O_{\P}(1)\tsr \cdots \tsr {\rm pr}_d^*\O_{\P}(1)$ on $\P^{\times d}$ via the $\k$-linear isomorphisms
$$H^0(\P^{\times d}, \O_{\P^{\times d}}(1))\cong H^0(\P, \O_{\P}(1))^{\tsr d}\cong ((V^*)^*)^{\tsr d}\cong V^{\tsr d}.$$
We define the \emph{$d$-th truncated point scheme for $A$},  $\U_d(A)$, to be the scheme of zeros of any basis for the space $\widetilde{I_d}=\{\widetilde{r}\ |\ r\in I_d\}$. Note that $\U_d(A)$ is independent of choice of basis for $\widetilde{I_d}$. Let $i_d:\U_d(A)\to \P^{\times d}$ be the inclusion morphism.
Note that since $A\in {\tt Alg}^1$, the map $\pi_A:{\mathbf T}(A_1)\to A$ is the identity in degree 1. Thus $I=\ker(\pi_A)$ satisfies $I=I_{\ge 2}$, so $\U_1(A)=\P$. 


For $1\le i\le j\le d$, there is a canonical projection to the $j-i+1$ factors of $\P^{\times d}$ in positions $i$ through $j$:
$${\rm pr}^{(d)}_{ij}:\P^{\times d}=\P\times \cdots \times \prod_{k=i}^j \P \times \cdots \times \P \to \prod_{k=i}^j \P.$$ 
These projections induce $\k$-linear maps $$\left({\rm pr}^{(d)}_{ij}\right)^*:H^0(\P^{\times(j-i+1)},\O_{\P^{\times(j-i+1)}}(1))\to H^0(\P^{\times d}, \O_{\P^{\times d}}(1)),$$ and under the identifications above, $$\left({\rm pr}^{(d)}_{ij}\right)^*(I_{j-i+1})\subset V^{\tsr i-1}\tsr I_{j-i+1}\tsr V^{\tsr d-j}\subset I_d.$$ If $Z'$ denotes the scheme of zeros of $\left({\rm pr}^{(d)}_{ij}\right)^*(I_{j-i+1})$, then by Lemma \ref{induced morphism on scheme of zeros}, 
there exists a morphism $Z'\to \U_{j-i+1}(A)$.  Since $\U_{d}(A) \subseteq Z'$, we restrict and obtain a morphism $\pi^{(d)}_{ij}:\U_d(A)\to \U_{j-i+1}(A)$  such that the following diagram commutes.

\centerline{
\xymatrix{
\U_d(A)\ar[r]^{i_d}\ar[d]^{\pi^{(d)}_{ij}} & \P^{\times d}\ar[d]^{{\rm pr}^{(d)}_{ij}}\\
\U_{j-i+1}(A)\ar[r]^{i_{j-i+1}} & \P^{\times (j-i+1)}
}}

~\\

\subsection{The categories ${\tt Sys}$ and ${\tt PSys}^n$}~\\ 


We denote by ${\tt Sys}$ the category whose objects are collections $$\mfZ=\{(Z_d,\pi^{(d)}_{ij})\}_{d\ge 1, \, 1\leq i \leq j \leq d},$$ where the $Z_d$ are schemes and the $\pi^{(d)}_{ij}:Z_d\to Z_{j-i+1}$ are scheme morphisms, compatible in the sense that 
$\pi^{(j-i+1)}_{i' j'} \circ \pi^{(d)}_{ij} = \pi^{(d)}_{i+i'-1 \ i+j'-1}$ for all $1 \leq i \leq j \leq d$ and all $1 \leq i' \leq j' \leq j-i+1$.
We further require the map $\pi_{1d}^{(d)}: Z_{d} \to Z_{d}$ to be the identity map for all $d\ge 1$.
A morphism $\varphi:\mfZ\to \mfZ'$
is a collection of scheme morphisms $\{\varphi_d:Z_d\to Z_d'\}_{d \geq 1}$ satisfying $(\pi^{(d)}_{ij})'\varphi_d=\varphi_{j-i+1}\pi^{(d)}_{ij}$ for all $d \geq 1$, $1\le i\le j\le d$. We refer to objects of ${\tt Sys}$ as \emph{systems}. 


\begin{lemma}
\label{functors between systems and schemes}
The functor ${\mathbf F}: {\tt Sys} \to {\tt Sch}$ given by ${\mathbf F}(\mfZ)=Z_1$ and $\mathbf{F}(\varphi)=\varphi_1$ has a right adjoint ${\mathbf G}: {\tt Sch} \to {\tt Sys}$, given on objects by 
$${\mathbf G}(X) = \{(X^{\times d}, {\rm pr}_{ij}^{(d)})\}_{d \geq 1, 1 \leq i \leq j \leq d}$$
where ${\rm pr}_{ij}^{(d)}$ are the canonical projections onto the factors in positions $i$ through $j$.
\end{lemma}

\begin{proof}
Let $X \in \Sch$. Define ${\mathbf G}(X) = \{(X_d, \pi_{ij}^{(d)})\}_{d \geq 1, 1 \leq i \leq j \leq d}$ as follows. Set $X_1 = X$ and $\pi_{11}^{(1)} = {\rm id}_X$. Fix some $d \geq 2$ and assume inductively that all $X_{d'}$ and $\pi_{ij}^{(d')}$ are constructed satisfying the desired compatibilities for all $d' < d$. Let $X_d = X_{d-1} \times X_1$. Let $\pi_{1d}^{(d)} = {\rm id}_{X_d}$. Let $\pi_{1, d-1}^{(d)}: X_d \to X_{d-1}$, $\pi_{dd}^{(d)}: X_d \to X_1$ be the canonical maps from the fiber product. If $j \leq d-1$, let $\pi_{ij}^{(d)} = \pi_{ij}^{(d-1)} \circ \pi_{1, d-1}^{(d)}$. If $i < d$, let $\pi_{id}^{(d)} = (\pi_{i, d-1}^{(d)}, \pi_{dd}^{(d)})$ be the unique map $X_d \to X_{d-i+1} = X_{d-i} \times X_1$ satisfying $\pi_{1, d-i}^{(d-i+1)} \circ \pi_{id}^{(d)} = \pi_{i, d-1}^{(d)}$ and $\pi_{d-i+1, d-i+1}^{(d-i+1)} \circ \pi_{id}^{(d)} = \pi_{dd}^{(d)}$. Then one checks that the $\pi_{ij}^{(d)}$ satisfy the required relations. Hence ${\mathbf G}(X) \in {\tt Sys}$. 

Let $f: X \to X'$ be a morphism in $\Sch$. Define ${\mathbf G}(f): {\mathbf G}(X) \to {\mathbf G}(X')$ as follows. Set ${\mathbf G}(f)_1 = f$. Fix some $d \geq 2$ and assume that all ${\mathbf G}(f)_{d'}: X_{d'} \to X'_{d'}$ have been defined and commute with the $\pi_{ij}^{(k)}$, $\pi_{ij}^{' (k)}$ in the appropriate sense. Let ${\mathbf G}(f)_d: X_d \to X'_d = X'_{d-1} \times X'_1$ be the map $({\mathbf G}(f)_{d-1} \circ \pi_{1, d-1}^{(d)}, {\mathbf G}(f)_1 \circ \pi_{dd}^{(d)})$. Then one checks that $f_{d}$ commutes with the $\pi_{ij}^{(k)}$, $\pi_{ij}^{' (k)}$. Hence ${\mathbf G}(f)$ is a morphism in ${\tt Sys}$. It is then straightforward to check that ${\mathbf G}: \Sch \to {\tt Sys}$ is a functor. 

Let ${\mathbf F}: {\tt Sys} \to \Sch$ be the ``forgetful" functor,  ${\mathbf F}(\mfZ) = Z_1$ for $\mfZ \in {\tt Sys}$ and ${\mathbf F}(\varphi: \mfZ \to \mfZ') = \varphi_1$. For any $X \in \Sch$ and $\mfZ \in {\tt Sys}$, define $$T: \Hom_{{\tt Sys}}(\mfZ, {\mathbf G}(X)) \to \Hom_{\Sch}({\mathbf F}(\mfZ), X), \quad T(\varphi) = \varphi_1.$$ It is straightforward to check that $T$ is a natural isomorphism, and so $({\mathbf F}, {\mathbf G})$ are an adjoint pair of functors. 

\end{proof}


Let $\P=\P^n_{\k}$ denote $n$-dimensional projective space over $\k$ and let $\Phi^n={\mathbf G}(\P)$. 
We reserve the notation $\Phi^n$ for this system. A \emph{projective system} 
is a pair $(\mfZ, \n)$ consisting of an object $\mfZ\in {\tt Sys}$ and a morphism $\n:\mfZ\to \Phi^n$, for some $n$, in ${\tt Sys}$ such that $\n_d$ is a closed immersion for all $d\ge 1$. 

For fixed $n \geq 0$, let ${\tt PSys}^n$ denote the category whose objects are projective systems $(\mfZ, \n:\mfZ\to \Phi^n)$. A morphism $(\mfZ, \n) \to (\mfZ', \n')$ in ${\tt PSys}^n$ is a morphism $\varphi:\mfZ\to \mfZ'$ in ${\tt Sys}$ satisfying $\n'\circ \varphi = \n$. 
Since the codomain of $\n_d$ and $\n'_d \circ \varphi_d$ is the fiber product $\P^{\times d}$, it is not hard to check that the condition on morphisms holds if and only if $\n'_1\circ \varphi_1 = \n_1$.


%

\begin{defn} 
\label{special systems}
For $A\in {\tt Alg}^1$, let $n=\dim A_1-1$ and $\P=\P(A_1^*)$. Then $\U_A=\{(\U_d(A), \pi^{(d)}_{ij})\}$ defined in Section \ref{Tps} is an object of ${\tt Sys}$ and the pair $(\U_A, i)$ is a projective system in ${\tt PSys}^{n}$, where $i_d:\U_d(A)\to \P^{\times d}$ are the inclusions. We refer to $(\U_A, i)$ as the \emph{truncated point system of $A$}. 

\end{defn}



\begin{rmk}  For any $n \geq 0$, let ${\tt Proj}^n$ denote the category whose objects are pairs $(X, f)$, where $X \in \Sch$ and $f: X \to \P^n_{\k}$ is a closed immersion in $\Sch$. A morphism $(X, f) \to (X', f')$ in ${\tt Proj}^n$ is a morphism $g: X \to X'$ in $\Sch$ such that $f' \circ g = f$. 
The functors ${\mathbf F}$ and ${\mathbf G}$ defined in Lemma \ref{functors between systems and schemes} extend in the obvious way to a pair of adjoint functors ${\mathbf F}: {\tt PSys}^n \to {\tt Proj}^n$ and ${\mathbf G}: {\tt Proj}^n \to {\tt PSys}^n$. 

~\\
\end{rmk}

\subsection{The functor $\mbfB:{\tt PSys}^n \to {\tt Alg}$}~\\

Motivated by the construction of a graded algebra $B$ from the collection of truncated point schemes associated to a graded algebra $A$ in \cite[Section 3.17]{ATVI}, we now define a contravariant functor $\mbfB:{\tt PSys}^n \to {\tt Alg}$. 

Fix an object $(\mfZ, \n: \mfZ \to \Phi^n) \in {\tt PSys}^n$. For ease of notation, in this subsection we  write ${\rm pr}_{ij}$ and $\pi_{ij}$ instead of ${\rm pr}^{(d)}_{ij}$ and $\pi^{(d)}_{ij}$ when the domain is clear from context. 
Let $\mcL_d=\n_d^*\O_{\P^{\times d}}(1)$ be the restriction of $\O_{\P^{\times d}}(1)$ to $Z_d$. Then there is a canonical isomorphism $$\mcL_{i+j}\cong \pi_{1,i}^*\mcL_i\tsr_{\O_{Z_{i+j}}} \pi_{i+1,i+j}^*\mcL_j$$ for all $i,j\ge 1$.

For $d\ge 1$, let $B_d = H^0(Z_d, \mcL_d)$. Then for any $i,j\ge 1$ we have $\k$-linear maps $\mu_{i,j}:B_i\tsr B_j\to B_{i+j}$ given by
\begin{align*}
B_i\tsr B_j &\xto{\widehat{\pi}_{1,i}\tsr \widehat{\pi}_{i+1,i+j}}H^0(Z_{i+j}, \pi_{1,i}^*\mcL_i)\tsr H^0(Z_{i+j}, \pi_{i+1,i+j}^*\mcL_j)\\
&\xto{\qquad\cup\qquad} H^0(Z_{i+j},\pi_{1,i}^*\mcL_i\tsr_{\O_{Z_{i+j}}} \pi_{i+1,i+j}^*\mcL_j)\\
&\xto{\qquad\cong\qquad} H^0(Z_{i+j},\mcL_{i+j})=B_{i+j},
\end{align*}
where $\widehat{\pi}_{1,i}$ and $\widehat{\pi}_{i+1,i+j}$ are the canonical induced maps.
The second map is the cup product, and the third is obtained by applying $H^0$ to the isomorphism above.
It is straightforward to check that, setting $B_0=\k$, the maps $\mu_{i,j}$ endow the space  $\mbfB(\mfZ)=\bigoplus_{d\ge 0} B_d$ with the structure of a unital, associative, graded algebra. 

If $\varphi:\mfZ\to \mfZ'$ is any morphism in ${\tt PSys}^n$, it follows that 
$(\pi_{1i})'\varphi_{i+j}=\varphi_{i}\pi_{1i}$ and $(\pi_{i+1,i+j})'\varphi_{i+j}=\varphi_{j}\pi_{i+1,i+j}$ for all $1\le i\le j$. This, together with the fact that cup product is compatible with pullbacks, implies that the direct sum of the canonical maps $\widehat{\varphi}_d:H^0(Z'_d,\mcL'_d)\to H^0(Z_d,\mcL_d)$ induces a morphism of graded algebras
$\mbfB(\varphi) : \mbfB(\mfZ')\to \mbfB(\mfZ).$ One easily checks that $\mbfB:{\tt PSys}^n \to {\tt Alg}$ is a contravariant functor.


\begin{ex}
\label{free algebra}
Consider the projective system $(\Phi^n, {\rm id}) \in {\tt PSys}^n$. Let $V$ be a vector space of dimension $n+1$. If $\P=\P(V^*)$, we have $$H^0(\P^{\times d}, \O_{\P^{\times d}}(1))=((V^*)^*)^{\tsr d}=V^{\tsr d}.$$ 
Thus $\mbfB(\Phi^n, {\rm id})={\mathbf T}(V)$ is the tensor algebra on $V$.
\end{ex}

\section{Realizing $\U_d(A)$ via Proj} 
\label{Td}


Let $A$ be an object of  ${\tt Alg}^1$. Realizing the scheme $\U_d(A)$ as  $\Proj {\mathbf U}_d(A)$, for a certain graded, commutative algebra ${\mathbf U}_d(A)$ will be useful in Sections \ref{structure} and \ref{examples} for studying the relationship between $A$ and $\mathbf{B}(\U_A)$. In this section we define a family of functors $\{{\mathbf U}_d\}_{d\ge 2}$ on the category ${\tt Alg}^1$, and then we use Lemma \ref{scheme of zeros as proj} to prove there is an isomorphism of schemes, $\U_d(A) \cong \Proj {\mathbf U}_d(A)$. We will prove that if $A$ is commutative, then for sufficiently large $d$, ${\mathbf U}_d(A)$ is isomorphic to the $d$-th Veronese subalgebra of $A$. This result will also be used in Section \ref{examples}.
~\\

\subsection{The functor ${\mathbf U}_d$}~\\ 
\label{T_d functor}

Let ${\tt CAlg}$ denote the category whose objects are commutative, graded algebras and whose morphisms are graded algebra maps.  For any integer, $d \geq 2$, we associate to an algebra $A\in {\tt Alg}^1$ certain \emph{multilinearization data}, which we use to define a functor ${\mathbf U}_d : {\tt Alg}^1 \to {\tt CAlg}$. 

Let $\Sym: {\tt Vect} \to {\tt CAlg}$ denote the symmetric algebra functor. 
 Let $A \in {\tt Alg}^1$. Let $V = A_1$ as a vector space. Recall that $\pi_A: {\mathbf T}(V) \to A$ denotes the canonical algebra morphism, and let $I = \ker \, \pi_A$. Consider the commutative algebra $S = \Sym(V)^{\tsr d}$, where the algebra structure is the usual one on a $d$-fold tensor product of algebras. Let $R = \bigoplus_{i \geq 0} [\Sym(V)_i]^{\tsr d}$, and note that $R$ is a subalgebra of $S$. Define a $\k$-linear map
$$\iota: V^{\tsr d} \to R, \quad v_1 \tsr \cdots \tsr v_d \mapsto v_1 \tsr \cdots \tsr v_d.$$ 
We refer to $(V, I, S, R, \i)$ as the  \emph{multilinearization data} associated to $A$. We note that the multilinearization data of $A$ depends on $d$.

We define ${\mathbf U}_d$ on objects of ${\tt Alg}^1$ by
 ${\mathbf U}_d(A) = R/(\iota(I_d))$, where $(\iota(I_d))$ is the ideal of $R$ generated by the linear space $\iota(I_d)$. Note that ${\mathbf U}_d({\mathbf T}(V))=R$.

Let $f: A \to A'$ be a morphism in ${\tt Alg}^1$. Let $(V', I', S', R', \i')$ be the multilinearization data associated to $A'$.
Recall that we have ${\mathbf T}(f_1)(I) \subseteq I'$, where ${\mathbf T}(f_1): {\mathbf T}(V) \to {\mathbf T}(V')$. Applying the functor $\Sym$ to the linear map $f_1: V \to V'$ yields an algebra map $\Sym(f_1): \Sym(V) \to \Sym(V')$, and this determines an algebra map $\Sym(f_1)^{\tsr d}: R \to R'$. 
It is clear that the following diagram commutes.
$$
\xymatrix{
V^{\tsr d} \ar[r]^{\ \iota} \ar[d]_{{\mathbf T}(f_1)} & R \ar[d]^{\Sym(f_1)^{\tsr d}} \\
(V')^{\tsr d} \ar[r]^{\ \iota'} & R'\\
}
$$
Hence $\Sym(f_1)^{\tsr d}(\iota(I_d)) \subseteq \iota'(I'_d)$. It follows that $\Sym(f_1)^{\tsr d}$ induces an algebra map ${\mathbf U}_d(A) \to {\mathbf U}_d(A')$ and we define ${\mathbf U}_d(f)$ to be this map.

It is straightforward to check that ${\mathbf U}_d$ is a functor from ${\tt Alg}^1$ to ${\tt CAlg}$. The following lemma, needed later, has an easy proof, which we omit. 
\begin{lemma}\label{T_n preserves surjecttions}
Let $A$, $A'$ be objects of ${\tt Alg}^1$. Let $d \geq 2$. If $f: A \to A'$ is a surjective morphism in ${\tt Alg}^1$, then ${\mathbf U}_d(f): {\mathbf U}_d(A) \to {\mathbf U}_d(A')$ is a surjective morphism in ${\tt CAlg}$.
\end{lemma}


We now show that $\Proj {\mathbf U}_d(A)$ is indeed a model for the schemes $\U_d(A)$.

\begin{prop}
\label{model for Gd}
Let $A\in {\tt Alg}^1$ and let $d\ge 2$. Then  $\U_d(A) \cong \Proj {\mathbf U}_d(A)$.
\end{prop}

\begin{proof}
Let $f\in V^{\tsr d}$ and $\P=\P(V^*)$. We may view $f$ as a global section of $\O_{\P^{\times d}}(1):={\rm pr}_1^*\O_{\P}(1)\tsr\cdots\tsr {\rm pr}_d^*\O_{\P}(1)$ via the canonical isomorphisms
$$H^0(\P^{\times d}, \O_{\P^{\times d}}(1))=H^0(\P, \O_{\P}(1))^{\tsr d}=((V^*)^*)^{\tsr d}=V^{\tsr d}.$$
It is a standard exercise (e.g. \cite[Exercise II.5.11]{Hart}) to show that if $\P^{\times d}$ is realized as $\Proj R$, for $R = \bigoplus_{i \geq 0} \Sym(V)_i^{\tsr d}$ as above, then the global section determined by $f$ via the identifications above corresponds to the image of $f$ under the map
$\z(R)_1:V^{\tsr d} \to H^0(\P^{\times d}, \O_{\P^{\times d}}(1))$ from the four-term sequence (\ref{four-term-sequence}) following the proof of Lemma \ref{scheme of zeros kills global sections}. The result follows from Lemma \ref{scheme of zeros as proj}.

\end{proof}

 For use  in Section 6, we now give another description of the algebra ${\mathbf U}_d(A)$.  Fix an integer $d \geq 2$. 
 Let $A \in {\tt Alg}^1$ and let $(V, I, S, R, \i)$ be the multilinearization data associated to $A$. 
 Recall 
 that ${\mathbf U}_d(A) = R/(\iota(I_d))$.  


For $v \in V$ and $0 \leq i \leq d-1$, let $v_i = 1 \tsr \cdots \tsr v \tsr \cdots \tsr 1 \in S$ and $\overline{v}_i=(0,\ldots, v,\ldots,0)\in V^{\oplus d}$, where $v$ is in position $i$ in each case. Then the $\k$-linear map $V^{\oplus d}\to S$ given by $\overline{v}_i\mapsto v_i$ induces a graded algebra isomorphism $\Sym(V^{\oplus d})\cong S$. Thus if $\mathcal B$ is a basis for $V$, then $S$ is freely generated as a commutative algebra by $\{b_i\ :\ b\in \mathcal B, 0\le i\le d-1\}$. For $v\in V$, the map $$\s(v_i) = \begin{cases} v_{i+1} \quad & 0 \leq i \leq d-2 \\ 0 \quad & i = d-1 .\end{cases}$$ is $\k$-linear, hence it determines a algebra endomorphism $\s: S \to S$.
Define an $\N^d$ grading on $S$ by $\deg(v_i) = (0, \ldots, 1, \ldots, 0)=:e_i$, where $1$ is in position $i$. 
Notice that $R$ is the subalgebra of $S$ generated by $S_{e_0+\cdots+e_{d-1}}=S_{(1, \ldots, 1)}$. 

For $2 \leq n \leq d$, define a $\k$-linear map $$\iota_n: V^{\tsr n} \to S, \quad v_1 \tsr \cdots \tsr v_n \mapsto v_1 \tsr \cdots \tsr v_n \tsr 1 \cdots \tsr 1.$$ Observe that $\iota = \iota_d$. 

\begin{defn}\label{multilinearization ideal}
Let $d \geq 2.$ If $(V, I, S, R, \i)$ is the  multilinearization data associated to $A\in {\tt Alg}^1$, we call the ideal 
of $S$ generated by the subspaces $$\s^m(\iota_n(I_n)), \quad 2 \leq n \leq d, \quad 0 \leq m \leq d-n$$
the \emph{multilinearization of $I$ in $S$}.
\end{defn}
We note that elements of the subspace $\s^m(\iota_n(I_n))$ are homogeneous of degree $e_{m,n}:=e_m+\cdots+e_{m+n-1}$.

\begin{lemma}\label{J = K cap R} Let $d\ge 2$ and $A\in {\tt Alg}^1$ with  multilinearization data $(V, I, S, R, \i)$. If $K$ is the multilinearization ideal of $I$ in $S$ and $J$ is the ideal of $R$ generated by $\iota(I_d)$, then $K \cap R=J$.
\end{lemma}

\begin{proof} 
It is obvious that $J \subseteq K \cap R$. To prove the other containment, we use the fact that  $S$ is graded by $\N^d$; note that $K$, $R$ and $J$ are all homogeneous with respect to this grading.
 It suffices to prove that $(K \cap R)_{(i, \ldots, i)} \subseteq J$ for all $i \geq 1$. Denoting multiplication in $S$ by $\cdot_S$, we have  
\begin{align*}
(K \cap R)_{(i, \ldots, i)} &= \sum_{n = 2}^d \sum_{m = 0}^{d-n} \s^m(\iota_n(I_n)) \cdot_S S_{(i,\ldots, i)-e_{m,n}} \\
&= \sum_{n = 2}^d \sum_{m = 0}^{d-n} \s^m(\iota_n(I_n)) \cdot_S \left(V^{\tsr m} \tsr \k^{\tsr n} \tsr V^{\tsr d-(m+n)}\right) \cdot_S R_{i-1} \\
&= \sum_{n = 2}^d  \left[\iota_d\left(\sum_{m=0}^{d-n} V^{\tsr m} \tsr I_n \tsr V^{d-(m+n)}\right)\cdot_S R_{i-1}\right] \\
&\subseteq  \sum_{n = 2}^d \iota(I_d) \cdot_S R_{i-1} \\
&\subseteq J.
\end{align*}

\end{proof}

Let $S' = \bigoplus_{i \geq 0} (S/K)_{(i, \ldots, i)}$. Then $S'$ is $\N$-graded by setting $S'_i = (S/K)_{(i, \ldots, i)}$. Furthermore, note that $S'$ is generated in degree $1$.

\begin{prop}\label{alt pres}
There is an isomorphism of graded algebras ${\mathbf U}_d(A) \cong S'$.
\end{prop}

\begin{proof}
By definition, ${\mathbf U}_d(A) = R/J$,  where $J$ is the ideal of $R$ generated by $\iota(I_d)$. Lemma \ref{J = K cap R} implies that the map $$R/J \to S/K, \quad r + J \mapsto r+K$$ is injective. Then the result follows by observing that the image of this map is $S'$. 

\end{proof}

\subsection{The functor ${\mathbf U}_d$ on commutative algebras}~\\

Let $A \in {\tt Alg}^1$ be a commutative algebra. Then we may write $A = {\mathbf T}(V)/(I+J)$, where $V = A_1$,  
$I=\la [V,V]\ra$ is the commutator ideal,
and $J$ is a homogeneous  ideal in the tensor algebra ${\mathbf T}(V)$. Though $J$ is not uniquely determined in this context, the degree(s) of the generators of $J/(I\cap J)$ are.   For any $d \geq 1$, let $A_{(d)} = \oplus_{n = 0}^{\infty} A_{nd}$ denote the $d$-th Veronese subalgebra of $A$. 

\begin{thm}\label{functor on commutative}
 Let $A \in {\tt Alg}^1$ be commutative and let $V=A_1$. Let $J\subset \mathbf{T}(V)$ be any homogeneous ideal such that $A={\mathbf T}(V)/(I+J)$, where $I=\la [V,V]\ra$, and such that $J$ is generated as an ideal by $J_{\le d}$ for $d\ge 2$. Then there is an isomorphism of $\N$-graded algebras ${\mathbf U}_d(A) \cong A_{(d)}$. 
\end{thm} 

\begin{proof}
Let $d \geq 2$ be as in the statement of the theorem. Let $T = \Sym(V^{\tsr d})$ and $R = \bigoplus_{i \geq 0} \Sym(V)_i^{\tsr d}$. Define $\k$-linear maps 
\begin{align*}
&\a: V^{\tsr d} \to T, \quad v_1 \tsr \cdots \tsr v_d \mapsto v_1 \tsr \cdots \tsr v_d, \\
&\iota: V^{\tsr d} \to R, \quad v_1 \tsr \cdots \tsr v_d \mapsto v_1 \tsr \cdots \tsr v_d.
\end{align*}
The universal property of the functor $\Sym$ affords a algebra morphism $\Sigma: T \to R$ such that $\iota = \Sigma \circ \a$. 

Consider the map $$\mu: R \to \Sym(V), \quad f_1 \tsr \cdots \tsr f_d \mapsto f_1 \cdots f_d.$$ It is easy to check that $\mu$ is a well-defined algebra homomorphism. Identify $A$ with $\Sym(V)/J$ and let $\pi: \Sym(V) \to A$ be the canonical quotient map. Then, composing $\pi$ with $\mu$ and abusing notation, we have a map $\mu: R \to A$. Observe that $\mu: R \to A$ surjects onto $A_{(d)}$. It is well known that $$I_d = \text{span}_{\k}\{v_1 \tsr \cdots \tsr v_d - v_{\s(1)} \tsr \cdots \tsr v_{\s(d)} : v_i \in V, \  \s \in \mathcal{S}_d\},$$ where $\mathcal{S}_d$ denotes the $d$-th symmetric group. Clearly $I_d$ is contained in the kernel of $\mu: R \to \Sym(V)$, and also $J_d$ is contained in the kernel of $\mu: R \to A$. Recalling that ${\mathbf U}_d(A) = R/(\iota(I_d+J_d))$, we see that there is an induced surjective algebra homomorphism $\mu: {\mathbf U}_d(A) \to A_{(d)}$.

Now we will show that $\Sigma: T \to R$ induces a surjective algebra homomorphism $\Sigma: A_{(d)} \to {\mathbf U}_d(A)$. To that end, consider the map $q \circ \Sigma: T \to {\mathbf U}_d(A)$, where $q: R \to {\mathbf U}_d(A)$ is the canonical quotient map. The ideal of $T$ generated by $\alpha(I_d)$ is in the kernel of $q \circ \Sigma$, hence there is an induced map $\Sigma: T/(\alpha(I_d)) \to {\mathbf U}_d(A)$. Consider the short exact sequence of vector spaces: 

$$
\xymatrix{
& 0  \ar[r] & I_d \ar[r] & V^{\tsr d} \ar[r] & V^{\tsr d}/I_d \ar[r] & 0 .\\
}
$$
It is well known, \cite[III.6, Proposition 4]{Bourbaki}, that the map $\Sym(V^{\tsr d}) \to \Sym(V^{\tsr d}/I_d)$ is surjective and $$\ker(\Sym(V^{\tsr d}) \to \Sym(V^{\tsr d}/I_d)) = (\alpha(I_d)).$$ Hence $$\Sym(V^{\tsr d}/I_d) \cong \Sym(V^{\tsr d})/(\alpha(I_d)).$$ So we identify $S/(\alpha(I_d))$ with $\Sym(V^{\tsr d}/I_d)$, and under this identification there is a surjective algebra homomorphism $\Sigma: \Sym(V^{\tsr d}/I_d) \to {\mathbf U}_d(A)$. We will abuse notation and write $v_1 \tsr \cdots \tsr v_d \in \Sym(V^{\tsr d}/I_d)$ for $v_i \in V$ with the understanding that  $v_1 \tsr \cdots \tsr v_d$ is a symmetric tensor. We denote multiplication in the ring $\Sym(V^{\tsr d}/I_d)$ by $\ast$. 

In \cite[p.\ 72]{Eisenbud-Reeves-Totaro} the authors give a presentation of the Veronese subalgebra $A_{(d)}$ of $A$ using coordinates. In our coordinate-free description, this presentation of $A_{(d)}$ is $\Sym(V^{\tsr d}/I_d)/K$, where $K$ is the ideal of $\Sym(V^{\tsr d}/I_d)$ generated by the following:
\begin{align*}
& (1) \text{ the spaces } J_i \tsr V^{\tsr d-i}, \ 2 \leq i \leq d, \\
& (2) \text{ for all }  v_i \in V, \ \s \in \mathcal{S}_{2d}, \text{ the elements}\\
&  (v_1 \tsr \cdots \tsr v_d) \ast (v_{d+1} \tsr \cdots \tsr v_{2d}) - (v_{\s(1)} \tsr \cdots \tsr v_{\s(d)}) \ast (v_{\s(d+1)} \tsr \cdots \tsr v_{\s(2d)}).
\end{align*}
We claim that $K$ is contained in $\ker(\Sigma: \Sym(V^{\tsr d}/I_d) \to {\mathbf U}_d(A))$. Firstly, note that $J_i \tsr V^{\tsr d-i}$ is a subset of the degree $1$ component of $\Sym(V^{\tsr d}/I_d)$ and $\Sigma: \Sym(V^{\tsr d}/I_d) \to {\mathbf U}_d(A)$ kills this space because ${\mathbf U}_d(A) = R/(\iota(I_d+J_d))$. Now consider some element 
$$b = (v_1 \tsr \cdots \tsr v_d)*(v_{d+1} \tsr \cdots \tsr v_{2d}) - (v_{\s(1)} \tsr \cdots \tsr v_{\s(d)})*(v_{\s(d+1)} \tsr \cdots \tsr v_{\s(2d)})$$
of $\Sym(V^{\tsr d}/I_d)$, $\s \in \mathcal{S}_{2d}$. Using the fact that the tensors here are symmetric, we may assume that in the term $(v_{\s(1)} \tsr \cdots \tsr v_{\s(d)})*(v_{\s(d+1)} \tsr \cdots \tsr v_{\s(2d)})$
that for all $1 \leq i \leq d$, $v_i$ occurs in position $i$ inside its factor: either $v_\s(1) \tsr \cdots \tsr v_\s(d)$ or $v_{\s(d+1)} \tsr \cdots \tsr v_{\s(2d)}$. Then, as an element of $R/(\iota(I_d+J_d))$, 
\begin{align*} 
\Sigma(b) &= (v_1 v_{d+1} \tsr \cdots \tsr v_d v_{2d} - v_{\s(1)} v_{\s(d+1)} \tsr \cdots \tsr v_{\s(d)} v_{\s(2d)}) + (\iota(I_d+J_d)) \\
&= (v_1 \tsr \cdots \tsr v_d) \cdot [v_{d+1} \tsr \cdots \tsr v_{2d} - v_{\tau(d+1)} \tsr \cdots \tsr v_{\tau(2d)}] + (\iota(I_d+J_d)),
\end{align*}
where $\cdot$ denotes the multiplication in $R$ and $\tau$ is some permutation of the set $\{d+1, ... , 2d\}$. Now note that $$v_1 \tsr \cdots \tsr v_d \cdot [v_{d+1} \tsr \cdots \tsr v_{2d} - v_{\tau(d+1)} \tsr \cdots \tsr v_{\tau(2d)}]$$ is an element of the ideal of $R$ generated by $\iota(I_d)$, and therefore $\Sigma(b) = 0$.

We have proved that the ideal $K$ is contained in the kernel of the homomorphism $\Sigma: \Sym(V^{\tsr d}/I_d) \to {\mathbf U}_d(A)$, so we obtain a map $\Sigma: \Sym(V^{\tsr d}/I_d)/K \to {\mathbf U}_d(A).$ Using \cite{Eisenbud-Reeves-Totaro} we know that there is an algebra isomorphism $A_{(d)} \cong \Sym(V^{\tsr d}/I_d)/K$. Hence we have an algebra map $\Sigma: A_{(d)} \to {\mathbf U}_d(A)$. Finally, consider the sequence of surjective algebra maps: 
$$
\xymatrix{
& A_{(d)} \ar[r]^{\Sigma} & {\mathbf U}_d(A) \ar[r]^{\mu} & A_{(d)}.\\
}
$$

It is clear that the composite $\mu \circ \Sigma$ is an isomorphism in degree $1$, so $\mu \circ \Sigma$ is the identity map on $A_{(d)}$. So $\Sigma$ is injective. Therefore $\Sigma: A_{(d)} \to {\mathbf U}_d(A)$ is an algebra isomorphism.

%
%
%
%
%
%
%

\end{proof}

\section{Relating the algebras $A$ and  $\mbfB(\U_A)$}
\label{structure}


In this section we begin by defining an algebra morphism $\t_A: A \to \mbfB(\U_A)$. We then characterize when $\t_A$ is injective or surjective in terms of certain local cohomology modules. Finally, we study $\ker \, \t_A$ in terms of annihilator ideals of truncated point modules.

For an algebra $A \in {\tt Alg}^1$ with $\dim A_1=n+1$, recall the definition of the truncated point system of $A$, $(\U_A, i)\in {\tt PSys}^{n}$ from Definition \ref{special systems}. Let $\mbfB(A) = \mbfB(\U_A)$. 

\begin{prop}[\cite{ATVI}]
\label{map A to B}
If $A\in {\tt Alg}^1$, 
then there is a graded algebra morphism $\t_A:A\to \mbfB(A)$, which is bijective in degree 1.
\end{prop}

\begin{proof}
Let $i:\U_A\to \Phi^n$ be the inclusion morphism in ${\tt PSys}^n$. By functoriality of $\mbfB$ and Example \ref{free algebra}, there is a graded algebra morphism $\mbfB(i):{\mathbf T}(A_1)\to \mbfB(A)$. It follows from Lemma \ref{scheme of zeros kills global sections} that $I$ =  $\ker(\pi_A: {\mathbf T}(A_1) \to A)$ is contained in $\ker \,\mbfB(i)$. Hence there is an induced graded algebra morphism $\t_A:A\to \mbfB(A)$, as desired.

Recall that since $A\in {\tt Alg}^1$, we have $I=I_{\ge 2}$ and hence $\U_1(A)=\P$. The inclusion $i_1:\U_1(A)\to \P$ is an isomorphism, hence $(\t_A)_1$ is an isomorphism.

\end{proof}

As illustrated in Section \ref{examples}, Example 3, the map $\t_A$ need not be injective nor surjective. Indeed, $\mbfB(A)$ need not be 1-generated. In a subsequent paper, we will clarify the relationship between $\mbfB(A)$ and the \emph{twisted homogeneous coordinate ring} associated to a ``triple'' $(E,\s,\mcL)$ consisting of a projective scheme $E$, a scheme automorphism $\s:E\to E$, and a line bundle $\mcL\in \mcO_E$-mod. As a consequence, we will show that when $A$ is a quadratic Artin-Schelter regular algebra of dimension 3, the twisted homogeneous coordinate ring of the triple associated to $A$ as in \cite{ATVI} is isomorphic to $\mbfB(A)$. In particular, $\t_A$ is always surjective when $A$ is AS-regular of dimension 3.

 Let $R={\mathbf U}_d({\mathbf T}(V))$ and recall from Section \ref{Td} the obvious map $\i:V^{\tsr d}\to R$. By definition of ${\mathbf U}_d$, this map induces a $\k$-linear isomorphism $\widetilde{\i}:A_d\to {\mathbf U}_d(A)_1$. Consider the following diagram, in which we have identified $\P^{\times d}\cong \Proj R$ and denoted the multilinearization isomorphism by $\widetilde{\cdot}:V^{\tsr d}\to H^0(\P^{\times d}, \O_{\P^{\times d}}(1))$.
$$
\xymatrix{
R_1 \ar[rr]^(0.4){\z(R)_1}\ar[ddd]_{{\mathbf U}_d(\pi_A)_1} && H^0(\P^{\times d}, \O_{\P^{\times d}}(1))\ar[ddd]^{\mbfB(\n)_d}\\
& V^{\tsr d}\ar[ul]^{\i}\ar[d]^{(\pi_A)_d}\ar[ur]^(0.4){\widetilde{\cdot}} & \\
& A_d\ar[rd]^{(\t_A)_d}\ar[ld]_{\widetilde{\i}} &\\
{\mathbf U}_d(A)_1\ar[rr]^(0.45){\z({\mathbf U}_d(A))_1}& & H^0(\U_d(A), \mcL_d)
}
$$
The outer square commutes by Lemma \ref{naturality of four-term sequence}. The left trapezoid commutes by properties of quotients, and the right trapezoid commutes by Proposition \ref{map A to B}. The upper triangle commutes by the discussion in the proof of Proposition \ref{model for Gd}. Since all vertical maps are zero on $I_d$, the lower triangle commutes as well.

\begin{thm}
\label{kernel-image via local cohomology}
Let $A\in {\tt Alg}^1$. For all $d\ge 2$, $\widetilde{\i}(\ker(\tau_A)_d)=H^0_{\mf m}({\mathbf U}_d(A))_1$, where ${\mf m}={\mathbf U}_d(A)_+$.
The map $(\t_A)_d$ is surjective if and only if $H^1_{\mf m}({\mathbf U}_d(A))_1=0$. Consequently,
\begin{enumerate}
\item $\t_A$ is injective if and only if $H^0_{\mf m}({\mathbf U}_d(A))_1=0$ for all $d\ge 2$,
\item $\mbfB(A)$ is 1-generated if and only if $H^1_{\mf m}({\mathbf U}_d(A))_1=0$ for all $d\ge 2$.
\end{enumerate}
In particular, if $\depth\ {\mathbf U}_d(A)\ge 2$ for all $d\ge 2$, then $A\cong  \mbfB(A)$.
\end{thm}

\begin{proof}
All of the statements preceding the last statement follow easily from the following facts: $\widetilde{\i}:A_d\to {\mathbf U}_d(A)_1$ is a $\k$-linear isomorphism, the following diagram is commutative with an exact bottom row
$$\xymatrix{
  &A_d \ar[d]^{\widetilde{\i}} \ar[rd]^{(\t_A)_d}  & & & \\
 0\to H^0_{\mf m}({\mathbf U}_d(A))_1  \ar[r] & {\mathbf U}_d(A)_1  \ar[r]^-{\z_1} & H^0(\U_d(A), \mcL_d) \ar[r] & H^1_{\mf m}({\mathbf U}_d(A))_1 \to 0  &
},$$
and $(\t_A)_1$ is a $\k$-linear isomorphism.
The last statement now follows immediately from the fact that $\depth\ {\mathbf U}_d(A)= \inf\{n : H^n_{\mf m}({\mathbf U}_d(A)) \ne 0\}$.

\end{proof}

The bijection between closed points of $\U_d(A)$ and isomorphism classes of truncated point modules of length $d+1$ was described in \cite[Proposition 3.9]{ATVI}. To establish our characterization of the kernel of $\t_A$, it is helpful to describe this correspondence in more detail.

If $M$ is a truncated point module for ${\mathbf T}(V)$ of length $d + 1$, we denote $L_i(M)= \Hom_{\k}(M_{i-1} ,M_i)$ for each $1\le i\le d$. Let $\rho:M\tsr {\mathbf T}(V)\to M$ be the module structure map. Let $\widetilde{\rho}_i(M):V\to L_i(M)$ be the map adjoint to $\rho_i:M_{i-1}\tsr V\to M_i$.
Since $M$ is cyclic, generated in degree 0, the maps $\widetilde{\rho}_i(M)$ are nonzero. Choosing bases for the graded components $M_i$, we may view the $\widetilde{\rho}_i(M)$ as equivalence classes in $\P(V^*)$, and the collection $\widetilde{\rho}(M)=\{\widetilde{\rho}_i(M)\}_{1\le i\le d}$ as an element of the cartesian product $\P(V^*)^{\times d}$.

Conversely, given a collection $\widetilde{\rho}=\{\widetilde{\rho}_i\in \P(V^*)\}_{1\le i\le d}$, for $0\le i\le d$, let $M'_i$ be a 1-dimensional $\k$-vector space with distinguished basis element $x_i$. For $1\le i\le d$, define $\k$-linear maps $\rho'_i:M'_{i-1}\tsr V\to M'_i$ by $x_{i-1}\tsr v\mapsto [\widetilde{\rho}_i](v)x_i$, where $[\widetilde{\rho}_i]$ is any representative of the equivalence class $\widetilde{\rho}_i$. Since $\mathbf{T}(V)$ is free, these maps uniquely determine the structure of a graded right $\mathbf{T}(V)$-module on $M(\widetilde{\rho})=\bigoplus_{i=0}^d M'_i.$ Since the maps $[\widetilde{\rho}_i]$ are nonzero, $M(\widetilde{\rho})$ is cyclic, generated in degree 0, hence it is a truncated right point module of length $d+1$. One readily checks that $M(\widetilde{\rho}(M))\cong M$ and $\widetilde{\rho}(M(\widetilde{\rho}))$ differs from $\widetilde{\rho}$ at most by rescalings of the $\widetilde{\rho}_i.$  This proves the following.

\begin{lemma} 
\label{points and modules0}
The maps $M\mapsto \widetilde{\rho}(M)$ and $\widetilde{\rho}\mapsto M(\widetilde{\rho})$ are inverse bijective correspondences between:
\begin{enumerate}
\item the set of isomorphism classes of truncated right point modules for ${\mathbf T}(V)$ of length $d + 1$, and
\item the $d$-fold cartesian product (of sets) $\P(V^*)^{\times d}$
\end{enumerate}
\end{lemma}

The correspondence of Lemma \ref{points and modules0} can be helpfully extended to include the more easily generalizable \emph{set of $\k$-valued points of the fibered product $\P(V^*)^{\times d}$,} which we now describe.

\begin{notn}
For the remainder of this section, let $V$ be a finite-dimensional vector space, let 
$S={\rm Sym}(V)$, and let $R={\mathbf U}_d({\mathbf T}(V))$. Let $Y= \Proj \k[t]$ be the scheme over $\k$ with one point. 
\end{notn}

If $M$ is a truncated point module for ${\mathbf T}(V)$ of length $d + 1$, then for $1\le i\le d$, let $\mcL_i(M)=\widetilde{L_i(M)}$ be the quasicoherent $\O_Y$-module associated to the trivial $\k[t]$-module $L_i(M)$. Composition induces an isomorphism of invertible sheaves $\mcL_1(M) \tsr\cdots\tsr\mcL_d(M) \cong \widetilde{L(M)}$ where $L(M) = \Hom_{\k}(M_0, M_d)$, hence
there is a canonical graded ring isomorphism $$\G_*(Y,\mcL_1(M))\circ\cdots\circ\G_*(Y,\mcL_d(M))\to\G_*(Y,\widetilde{L(M)}),$$ where $\circ$ denotes the Segre product of graded rings. Henceforth, we suppress the scheme $Y$ and write $\G_*(\mcL_i(M))$ for $\G_*(Y,\mcL_i(M))$.
%
Note that $\G_*(\mcL_i(M))\cong \k[t]$ and hence $\G_*(\mcL_i(M))_{(0)}=\k$.

\begin{const}\label{point pair from maps} ($\Delta, \nabla$)
Given $M$, a truncated point module for ${\mathbf T}(V)$ of length $d + 1$, let $L_i=L_i(M)$, $L=L(M)$, $\mcL_i=\mcL_i(M)$, and $\widetilde{\rho}=\widetilde{\rho}(M)$. There are surjective homomorphisms of graded rings $\widetilde{\psi}_i:S\to  \G_*(\mcL_i)$  given on generators by $\widetilde{\psi}_i(v)=\widetilde{\rho}_i(v)$. 
Since $R=S\circ\cdots\circ S$, the $\widetilde{\psi_i}$ induce a canonical ring map 
$$\widetilde{\psi}:R\to  \G_*(\mcL_1)\circ\cdots\circ\G_*(\mcL_d)\cong \G_*(\widetilde{L})$$
given on generators by $v_1\tsr\cdots\tsr v_d\mapsto \widetilde{\rho}_d(v_{d})\circ\cdots\circ\widetilde{\rho}_1(v_{1})$. This map is not zero since the $\widetilde{\rho}_i$ are nonzero, so $\mathfrak p=\ker \widetilde{\psi}\in \Proj R$ is a closed point, and $\widetilde{\psi}$ localizes to a nonzero local homomorphism of local rings $\psi:R_{(\mathfrak p)}\to \G_*(\widetilde{L})_{(0)}=\k$. We define $\Delta(M)=(\mf p, \psi)$. 

Observe that if $v\in R_1$ such that $v\notin\mathfrak p$, then for any $w\in R_1$, we have
$\psi(w/v)=\widetilde{\psi}(w)/\widetilde{\psi}(v)$. This implies that any nonzero rescaling of the $\widetilde{\rho}_i$ yields the same pair $(\mf p, \psi)$. In particular, $\Delta(M)=\Delta(M(\widetilde{\rho}(M)))$.

Conversely, given $(\mathfrak q, \varphi)$, where $\mf q$ is a closed point in $\Proj R$ and $\varphi: R_{(\mathfrak q)}\to\k$ is a nonzero local homomorphism of local rings, let $v=v_1\tsr\cdots\tsr v_d\in R_1$ such that $v\notin\mathfrak q$. Define 
$$\varphi'_i:S_{(v_i)}\xto{j_i} S_{(v_1)}\tsr\cdots\tsr S_{(v_d)}\xto{\g} R_{(v)}\to R_{(\mathfrak q)}\xto{\varphi}\k$$
where $j_i$ includes the $i$-th tensor factor and $\g$ is the canonical ring isomorphism.  
For $0\le i\le d$, define $\widetilde{\rho}'_i:V\to \k$ by $\widetilde{\rho}'_i(w)= \varphi'_i(w/v_i)$.
 Define $\nabla(\mathfrak q, \varphi) = M(\widetilde{\rho}')$ where $\widetilde{\rho}'=\{\widetilde{\rho}'_i\}_{1 \leq i \leq d}$. 

\end{const}

Recall that the data $(\mathfrak q, \varphi)$ in Construction \ref{point pair from maps} uniquely determines a scheme morphism $Y\to \Proj T$, hence they may be viewed as the data of a $\k$-valued point of $\Proj T$. Such pairs are important to the discussion below, so we make the following definition.

 \begin{defn}\label{pointPair}  Let $T$ be a standardly graded commutative $\k$-algebra. 
If  $\mathfrak p$ is a closed point of $\Proj T$ and $\psi:T_{(\mathfrak p)}\to \k$ is a nonzero local homomorphism of local rings, we call $(\mathfrak p, \psi)$ a \emph{point pair} for $\Proj T$.

 If $M$ is a truncated right point module for ${\mathbf T}(V)$, the pair $(\mathfrak p, \psi)=\Delta(M)$ described in Construction \ref{point pair from maps} will be referred to as \emph{the point pair for $M$}. 
 \end{defn}


We show that point pairs for $R$ determine point modules for ${\mathbf T}(V)$, up to isomorphism, extending Lemma \ref{points and modules0}.

\begin{prop} 
\label{points and modules}
There are bijective correspondences between the following:
\begin{enumerate}
\item the set of isomorphism classes of truncated right point modules for ${\mathbf T}(V)$ of length $d + 1$, 
\item the $d$-fold cartesian product (of sets) $\P(V^*)^{\times d}$, and
\item the set of point pairs $(\mathfrak p, \psi)$ for $\Proj R$.
\end{enumerate}
\end{prop}

By the remark above, (3) is equivalent to (3') \emph{the set of scheme morphisms $Y\to \Proj R$}.

\begin{proof} By Lemma \ref{points and modules0}, it suffices to show that $\Delta$ and $\nabla$ are inverse bijections. 



First we consider $\nabla \circ \Delta$. Let $M$ be a truncated point module of length $d+1$ for $\mathbf{T}(V)$ and let $\widetilde{\rho} = \{\widetilde{\rho}_i(M)\}_{1 \leq i \leq d}$. Set $(\mf p, \psi) = \Delta(\widetilde{\rho})$. 
We claim that $\nabla(\mf p, \psi)\cong M$. To see this, refer to Construction \ref{point pair from maps} and observe that $$\psi'_i(w/v_i)=\psi(v_1\tsr\cdots\tsr v_{i-1}\tsr w\tsr v_{i+1}\tsr\cdots\tsr v_d/v).$$ Furthermore, $\widetilde{\rho}'_i(w)= \psi'_i(w/v_i)=\widetilde{\rho}_i(M)(w)/\widetilde{\rho}_i(M)(v_i)$
 for all $w$, hence the $\widetilde{\rho}'_i$ recover the $\widetilde{\rho}_i(M)$ up to rescaling. It follows that $M(\widetilde{\rho}')\cong M$.

We now show that $\Delta \circ \nabla$ is the identity. Let $(\mathfrak q, \varphi)$ be an arbitrary point pair and set $M(\widetilde{\rho}) = \nabla(\mathfrak q, \varphi)$. Recall that the construction of $M(\widetilde{\rho})$ includes a choice of basis for each graded component. Let $L= \Hom_{\k}(M(\widetilde{\rho})_0, M(\widetilde{\rho})_d)$ and consider $\widetilde{\psi}:R\to \G_*( \widetilde{L})$, the ring map associated to $M(\widetilde{\rho})$, described in the construction of $\Delta(M(\widetilde{\rho}))$. We claim that $\mathfrak q=\ker\widetilde{\psi}$. Since $\mathfrak q$ is a closed point of $\Proj R$, and $\widetilde{\psi}$ is surjective, it suffices to show $\mathfrak q\subseteq\ker\widetilde{\psi}$.  

Let $f\in R_d$ be homogeneous and write $f=\sum w_{i_1}\cdots w_{i_d}$ for $w_{i_j}\in R_1$. We may assume the $w_{i_j}$ are pure tensors and denote the $k$-th tensor component $(w_{i_j})_k$. Then for the basis elements $x_0\in M(\widetilde{\rho})_0$ and $x_d\in M(\widetilde{\rho})_d$, unravelling definitions yields
\begin{align*}
\widetilde{\psi}(w_{i_j})(x_0)&=\widetilde{\rho}_d((w_{i_j})_d)\circ\cdots\circ \widetilde{\rho}_1((w_{i_j})_1)(x_0)\\
&=\varphi'_1((w_{i_j})_1/v_1)\cdots\varphi'_d((w_{i_j})_d/v_d)x_d\\
&=\varphi(\g(j_1((w_{i_j})_1/v_1)))\cdots\varphi(\g(j_d((w_{i_j})_d/v_d)))x_d\\
&=\varphi(\g((w_{i_j})_1/v_1\tsr\cdots\tsr (w_{i_j})_d/v_d))x_d\\
&=\varphi(w_{i_j}/v)x_d
\end{align*}
It follows that $\widetilde{\psi}(f)(x_0)=\varphi(f/v^d)x_d$. In particular, if $w\in R_1$, $\psi(w/v)=\varphi(w/v)$, so $\psi=\varphi$. And if $f\in \mathfrak q$, then $f/v^d\in\ker \varphi$ hence $f\in \ker\widetilde{\psi}$. We conclude that $\mathfrak q=\ker\widetilde{\psi}$.


We have established that $\Delta$ and $\nabla$ are inverse bijections. 


\end{proof}


\begin{prop}
\label{annihilator-prime}
Let $M$ be a truncated right point module for ${\mathbf T}(V)$ of length $d + 1$. 
Let $(\mathfrak p, \psi)$ be the point pair for $M$. 
An element $f\in {\mathbf T}(V)_d$ annihilates $M$ if and only if $\i(f)\in \mathfrak p$.
\end{prop}

\begin{proof} 
Let $\widetilde{\rho}(M)=\{\widetilde{\rho}_i\}_{1\le i\le d}$ as in the discussion preceding Lemma \ref{points and modules0}. For $f\in {\mathbf T}(V)_d$, write $f=\sum v_{i_1}\tsr\cdots\tsr v_{i_d}$, where $v_{i_j}\in V$. Then for any $m\in M_0$,
$$\rho(m\tsr f)=\sum \left(\widetilde{\rho}_d(v_{i_d})\circ\cdots\circ\widetilde{\rho}_1(v_{i_1})\right)(m)=
\widetilde{\psi}(\i(f))(m),$$
where $\widetilde{\psi}:R\to \G_*(\widetilde{L})$ is the map from Construction \ref{point pair from maps} that localizes to $\psi$.  
%
%
By construction, $\mathfrak p=\ker \, \widetilde{\psi}$, and the result follows.

\end{proof}

We can now record a version of Proposition \ref{points and modules} for $\Proj {\mathbf U}_d(A)$, our model for $\U_d(A)$ via Proposition \ref{model for Gd}. 
 
\begin{prop}
\label{passing to A}
Let $A\in {\tt Alg}^1$. Write $A={\mathbf T}(V)/I$ for $I = \ker \pi_A$, where $\pi_A:{\mathbf T}(A_1)\to A$. There are bijective correspondences between the following.
\begin{enumerate}
\item isomorphism classes of truncated right point modules for $A$ of length $d+1$,
\item point pairs $(\mathfrak p, \psi)$ for $\Proj R$ such that $\i(I_d)\subseteq \mathfrak p$, 
\item point pairs $(\mathfrak q, \varphi)$ for $\Proj U_d(A)$. 
\end{enumerate}
\end{prop}

\begin{proof} The bijection between (1) and (2) is a consequence of Propositions \ref{points and modules} and \ref{annihilator-prime}. We establish a bijection between (2) and (3).

Let $(\mathfrak p, \psi)$ be as in (2) and let $M=\nabla(\mathfrak p, \psi)$ be the corresponding truncated right point module. Since $\i(I_d)\subseteq \mathfrak p=\ker\widetilde{\psi}$, the map  $\widetilde{\psi}$ factors through ${\mathbf U}_d(A)$ as $\widetilde{\psi}_A:{\mathbf U}_d(A)\to \G_*(\widetilde{L})$. Since $M$ is a cyclic $A$-module, $\widetilde{\psi}_A$ is not zero, hence  $\mathfrak q_A=\ker\widetilde{\psi}_A$ is a closed point for $\Proj {\mathbf U}_d(A)$, and $\widetilde{\psi}_A$ induces a nonzero local homomorphism of local rings $\psi_A:{\mathbf U}_d(A)_{(\mathfrak q_A)}\to \k.$

Conversely, given a point pair $(\mathfrak q, \varphi)$ for  $\Proj {\mathbf U}_d(A)$, let $\mathfrak p={\mathbf U}_d(\pi_A)^{-1}(\mathfrak q)\in \Proj R$. Then $\mathfrak p$ is a closed point such that $\i(I_d)\subseteq\mathfrak p$,
and the composite map $\psi: R_{(\mathfrak p)}\xto{{\mathbf U}_d(\pi_A)_{(\mathfrak p)}} {\mathbf U}_d(A)_{(\mathfrak q)}\xto{\varphi}\k$ is defined. 
It is clear that this is inverse to the construction of the previous paragraph. 


\end{proof}

\begin{prop}
\label{annihilator-prime-A}
Let $A\in {\tt Alg}^1$. Let $M$ be a truncated right point module for $A$ of length $d + 1$. Let $(\mathfrak q, \varphi)$ be the point pair for $M$. An element $f\in A_d$ annihilates $M$ if and only if $\widetilde{\i}(f)\in \mathfrak q$.
\end{prop}

\begin{proof}
Write $A={\mathbf T}(V)/I$ where $I=\ker \pi_A$ for $\pi_A:{\mathbf T}(A_1)\to A$.  Let $f\in A_d$ and $f'\in T_d$ such that $\pi_A(f')=f$. By Proposition \ref{passing to A}(2), $(\mathfrak q, \varphi)$ corresponds to a point pair $(\mathfrak p, \psi)$ for $\Proj R$ with $\i(I_d)\subseteq\mathfrak p$.
%
%

Since ${\mathbf U}_d(\pi_A)\circ\i=\widetilde{\i}\circ\pi_A$, we have $\widetilde{\i}(f)\in \mathfrak q$ if and only if  $\i(f')\in \mathfrak p$, which is equivalent to $f'$ annihilating $M_{{\mathbf T}(V)}$ by Proposition \ref{annihilator-prime}. Clearly $f'$ annihilates $M_{{\mathbf T}(V)}$  if and only if $f$ annihilates $M_A$.

\end{proof}

For $T$  a standardly graded commutative algebra, let ${\rm Nil}(T)$ denote the intersection of all relevant homogeneous prime ideals of $T$.

\begin{thm}
\label{annihilating every point module} 
Let $A\in {\tt Alg}^1$. The following are equivalent:
\begin{enumerate}
\item $f\in A_d$ annihilates every truncated right point module for $A$ of length $d+1$.
\item $\widetilde{\i}(f)\in \mathfrak q$ for every closed point $\mathfrak q\in \Proj {\mathbf U}_d(A)$.
\item $\widetilde{\i}(f)$ is in ${\rm Nil}({\mathbf U}_d(A))$.
\end{enumerate}
\end{thm}

\begin{proof}
The equivalence of (1) and (2) follows from Propositions \ref{passing to A} and \ref{annihilator-prime-A}. It is trivial to see that (3) implies (2). For (2) $\Rightarrow$ (3), if $f\in {\mathbf U}_d(A)$ and $f\notin {\rm Nil}({\mathbf U}_d(A))$, then ${\mathbf U}_d(A)_{(f)}$ is nonzero, hence $D_+(f)$ contains a closed point, so there is some closed point of $\Proj {\mathbf U}_d(A)$ that doesn't contain $f$.


\end{proof}

Let 
$$J_d = \{f \in A_d : Mf = 0 \text{ for every truncated point module $M$ of length $d+1$}\}$$ 
 be the set of degree-$d$ elements of $A$ that annihilate every truncated right point module for $A$ of length $d+1$. We have the following characterization of the kernel of $\t_A$.

\begin{cor}
\label{kernel of tau}
For all $d\geq 2$, $\ker (\tau_A)_d\subseteq J_d$, with equality if $\U_d(A)$ is reduced.
\end{cor}

\begin{proof}
If $f\in \ker (\tau_A)_d$, then by Proposition \ref{kernel-image via local cohomology}, we have $\widetilde{\i}(f)\in \ker \z({\mathbf U}_d(A))_1=H_{\mathfrak m}^0(\mathscr T_d(A))_1.$  By Theorem \ref{annihilating every point module},  ${\rm Nil}({\mathbf U}_d(A))_1=\widetilde{\i}(J_d)$. It is easy to prove
that $H_{\mathfrak m}^0({\mathbf U}_d(A))_1\subseteq{\rm Nil}({\mathbf U}_d(A))_1$, with equality when $\U_d(A)\cong \Proj {\mathbf U}_d(A)$ is reduced. Since $\widetilde{\i}$ is a $\k$-linear isomorphism, the result follows.

\end{proof}

\begin{ex}
\label{Ore case 4} The inclusion of Corollary \ref{kernel of tau} may be strict when  $\U_d(A)$ is not reduced.
The algebra $A=\k\la x,y,z\ra/\la xy-yx, zy, zx-yz\ra$ is $H(0)$ in the classification of \cite[Theorem 5.2]{CG1}. Let  $K$ be the multilinearization of $\la xy-yx, zy, zx-yz\ra$ (see Definition \ref{multilinearization ideal}). 
It is straightforward to check that when $d=3$, a Gr\"obner basis for $K$ consists of
$$y_ix_{i+1}-x_iy_{i+1},\ z_iy_{i+1},\ z_ix_{i+1}-y_iz_{i+1},\quad 0\le i\le 1$$
$$y_0^2y_1z_2,\ y_i^2z_{i+1},\ y_iy_{i+1}z_{i+1},\ x_iy_{i+1}^2z_{i+1},\qquad 0\le i\le 1$$
Let $m=\widetilde{\i}(xyz)=x_0y_1z_2\in {\mathbf U}_3(A)$. Then the Gr\"obner basis implies $m^2=x_0^2y_1^2z_2^2=0$, 
so $m\in {\rm Nil}({\mathbf U}_3(A))$, and hence $xyz \in J_3$ by Theorem \ref{annihilating every point module}. Now let $m_x=\widetilde{\i}(x^3)=x_0x_1x_2$. It is readily checked using the Gr\"obner basis that $mm_x^n\neq 0$ for all $n>0$. Therefore $m\notin H^0_{\mathfrak m}({\mathbf U}_3(A))$, and it follows that $xyz\notin\ker(\tau_A)_3$.

\end{ex}

\section{Computing Hilbert series of $\mbfB(A)$ using geometry}
\label{Hilbert series}

The four-term sequence used in the proof of Proposition \ref{kernel-image via local cohomology} is of limited use in computing the Hilbert series of $\mbfB(A)$. It is often the case that $\ker\ \t_A$ is nonzero, and in such cases $\depth\ {\mathbf U}_d(A)=0$ for $d$ sufficiently large. 
Furthermore, the condition $H^1_{\frak m}({\mathbf U}_d(A))_1=0$ often becomes impractical to check directly for large $d$. However, in some cases, it is possible to exploit the geometry of the schemes $\U_d(A)$ to compute the Hilbert series of $\mbfB(A)$. An approach to this calculation via normalization was outlined for a particular family of schemes in \cite{Walton1}, also see \cite{Walton2}. In this section, we show the normalization approach applies when the $\U_d(A)$ are reduced schemes whose normalization map $\v$ is an isomorphism off a discrete set of singular points, assuming these singular points satisfy a certain separability condition. 


\begin{ass}
\label{ass}
Throughout this section, we assume $\k$ is an algebraically closed field, and $X$ is a reduced, Noetherian scheme whose singular locus ${\rm Sing}(X)$ is discrete (hence finite). We furthermore assume that each irreducible component $W_i\subset X$ is the scheme-theoretic image of a closed immersion $f_i:W'_i \to X$ with $W'_i$ normal.
\end{ass}

Let $\v:X' \to X$ be the normalization. By Assumptions \ref{ass}, we may take $(X',\v)=\left(\coprod W'_i, \coprod f_i\right)$ where the structure maps $\psi_i:W'_i\to X'$ satisfy $\v\circ \psi = f_i$. Since $\v = \coprod f_i$, we have $\v_*\O_{X'} = \bigoplus (f_i)_*\O_{W'_i}$. Since $X$ is reduced, the induced morphism $\v^{\#}:\O_X\to \v_*\O_{X'}$ is injective. Since $\v$ is an isomorphism over the open, dense regular locus, we have a short exact sequence
$$0\to \O_X\xto{\v^{\#}} \v_*\O_{X'} \xto{\rho} \mc{Q}\to 0,$$
where $\mc{Q}=\v_*\O_{X'}/\O_X$ is supported on ${\rm Sing}(X)$. Since ${\rm Sing}(X)$ is discrete, $\mc{Q}$ is a (finite) direct sum of skyscraper sheaves supported on ${\rm Sing}(X)$.
Tensoring the sequence above with $\O_X(1)$, and adopting the usual notation $\mc{F}(1):=\mc{F}\tsr_{\O_X}\O_X(1)$ for an $\O_X$-module $\mc{F}$, yields
$$0\to \O_X(1)\xto{\v^{\#}\tsr 1} \v_*\O_{X'}(1) \xto{\rho\tsr 1} \mc{Q}(1)\to 0.$$
We note that, by the projection formula, there is an isomorphism of $\O_X$-modules
$$\a: \bigoplus (f_i)_*f_i^*(\O_{X}(1)) \xto{\cong} \v_*\O_{X'}(1) .$$

While the notation $\mc{F}(1)$ is unambiguous when $\mc{F}=\O_X$, the same cannot be said of the expression $f_i^*\mc{F}(1)$. As $f_i^*(\O_X(1))$ makes later expressions somewhat tedious to parse, we henceforth abbreviate $\mc{F}_i=f_i^*(\O_X(1))$.

For a scheme $Y$ and an $\O_Y$-module $\mc{F}$, we write $h^0_Y(\mc{F}) = \dim_{\k} H^0(Y,\mc{F}).$

\begin{lemma}
\label{dimension count}
Let $X$ be a scheme satisfying Assumptions \ref{ass}. If $H^0(\rho\tsr 1)$ is surjective, then 
$$h^0_X(\O_X(1)) = \sum h^0_{W'_i}(\mc{F}_i) - \sum_{p\in {\rm Sing}(X)} \dim_{\k} \mc{Q}_p$$
\end{lemma}

\begin{proof}
Since $H^0(\rho\tsr 1)$ is surjective, the sequence 
$$0\to H^0(X,\O_X(1))\xto{H^0(\v^{\#}\tsr 1)} H^0(X,\v_*\O_{X'}(1)) \xto{H^0(\rho\tsr 1)} H^0(X,\mc{Q}(1))\to 0$$
is exact. Since $H^0(\a)$ is an isomorphism, we have
$$H^0(X,\v_*\O_{X'}(1))\cong H^0(X,\bigoplus (f_i)_*\mc{F}_i)\cong \bigoplus H^0(W'_i,  \mc{F}_i).$$
Finally, since $\O_X(1)$ is an invertible $\O_X$-module and since $\mc{Q}$ is a sum of skyscraper sheaves, $\mc{Q}(1)\cong \mc{Q}$ as $\O_X$-modules. The result follows from a dimension count of the short exact sequence.

\end{proof}

\begin{notn}
\label{trivializations} Let $x\in X$.
Since $f_i$ is a closed immersion, there is at most one $w\in W'_i$ such that $f_i(w)=x$. If such an element exists, we denote it $x'_i$, or just $x'$ when there is no confusion about the component in question.

Let $I_x=\{i : f_i^{-1}(x)\text{ is nonempty}\}$.
For each $i\in I_x$, since $f_i$ is a closed immersion, the canonical $\O_{X,x}$-module map $\g^i_{x,x'}:((f_i)_*\mc{F}_i)_x\to (\mc{F}_i)_{x'_i}$ is an isomorphism. (Note that $\O_{X,x}$ acts on $(\mc{F}_i)_{x'_i}$ via $(f_i^{\#})_{x'_i}$.) Let $\g_{x,x'} = \bigoplus_{I_x} \g^i_{x,x_i'}$. Then $\g_{x,x'}$ is also an isomorphism.

\end{notn}

\begin{defn}
\label{separable singularities}
Assume $X$ satisfies Assumptions \ref{ass}. We say the normalization map $\v:X'\to X$ \emph{separates singularities} if there 
exists a function $\t:{\rm Sing}(X) \to \bigcup H^0(W'_j,\mc{F}_j)$ with the property that for all $p\in {\rm Sing}(X)$, if $s=\t(p)\in H^0(W'_i,\mc{F}_i),$ then $s_{q'}\in \mf{m}_{\O_{W'_i,q'}}(\mc{F}_i)_{q'}$ if and only if $q'\neq p'$.
\end{defn}

We note that if $|W_i\cap {\rm Sing}(X)|\le 2$ for every component $W_i$, then $\v:X'\to X$ separates singularities, since $\O_{W'_i}(1)$ separates points. 

\begin{prop}
\label{constructing sections}
Assume $X$ satisfies Assumptions \ref{ass} and that $\v$ separates singularities. If $p\in {\rm Sing}(X)$ satisfies $\dim_{\k} \mc{Q}_p = 1$, then there exists a section $t\in\im\ H^0(\rho\tsr 1)$ such that ${\rm Supp}(t)=\{p\}$ and $t_p$ generates $\mc{Q}(1)_p$ as an $\O_{X,p}$-module.
\end{prop}

\begin{proof}
Let $i\in I_p$ and $s\in H^0(W'_i,\mc{F}_i)$ satisfy the condition of Definition \ref{separable singularities}. Extend $s$ to a section $$\overline{s}\in \bigoplus H^0(W'_i, \mc{F}_i) = H^0\left(X, \bigoplus (f_i)_*\mc{F}_i\right)$$ by setting $\overline{s}=(0,\ldots, s,\ldots, 0)$, where $s$ is in the $i$-th position. Let $s' =H^0(\a)(\overline{s})\in H^0(X, \v_*\O_{X'}(1))$ and define $t=H^0(\rho\tsr 1)(s')$.

It is clear by construction that $t_q=0$ for $q\notin W_i$. For any $q\in W_i\cap {\rm Sing}(X)$ and $q'=f_i^{-1}(q)$ we have the following sequence of $\O_{X,q}$-module maps
$$(\mc{F}_i)_{q'} \xto{\i_i}\bigoplus_{I_q} (\mc{F}_i)_{q'} \xto{\g^{-1}_{q,q'}} \bigoplus_{I_q} ((f_i)_*\mc{F}_i)_q\xto{\a_q}  (\v_*\O_{X'}(1))_q\xto{\rho_q\tsr 1}\mc{Q}(1)_q$$
where $\i_i$ is the canonical inclusion. Recall that $\rho_q\tsr 1$ is surjective. Since the $\O_{X,q}$-module structure on each $(\mc{F}_j)_{q'}$ is induced by the local homomorphism $(f_j^{\#})_{q'}:\O_{X,q}\to \O_{W'_j,q'}$, it follows that $s_{q'}\in \mf{m}_{\O_{W'_j,q'}}(\mc{F}_j)_{q'}$ if and only if $t_q\in \mf{m}_{\O_{X,q}}\mc{Q}(1)_q$.
By assumption, $s_{q'}\in \mf{m}_{\O_{W'_i,q'}}(\mc{F}_j)_{q'}$ if and only if $q\neq p$. Since $s_{p'}\notin \mf{m}_{\O_{W'_i,p'}}(\mc{F}_j)_{p'}$, it follows that $t_p\neq 0$ and $p\in {\rm Supp}(t)$. Finally, since $\mc{Q}(1)_q\cong \mc{Q}_q$ and since we assume $\dim_{\k} \mc{Q}_q=1$, the Nakayama lemma implies $\mf{m}_{\O_{X,q}}\mc{Q}(1)_q=0$. Consequently, $t_q=0$ if $q\neq p$, and $t_p$ generates $\mc{Q}(1)_p$. The result follows. 

\end{proof}

\begin{thm}\label{main hs theorem}
Let $X$ be a scheme satisfying Assumptions \ref{ass}. If $\v$ separates singularities and $\dim_{\k} \mc{Q}_p=1$ for all $p\in {\rm Sing}(X)$, then $H^0(\rho\tsr 1)$ is surjective and
$$h^0_X(\O_X(1)) = \sum_i h^0_{W'_i}(\mc{F}_i) - |{\rm Sing}(X)|.$$
\end{thm}

\begin{proof} If ${\rm Sing}(X)$ is empty, then $X$ is regular, hence normal, and $\v$ is an isomorphism. The conclusions are trivial in this case.
 
If ${\rm Sing}(X)$ is nonempty, enumerate the singular points ${\rm Sing}(X)=\{p_1,\ldots, p_{\ell}\}$. By Proposition \ref{constructing sections}, for each $1\le j\le \ell$, there exists a nonzero section $t_j \in H^0(X, \mc{Q}(1))$ such that ${\rm Supp}(t_j)=\{p_j\}$. The $t_j$ are linearly independent, since they are nonzero and their supports are disjoint. By assumption, $\dim_{\k} \mc{Q}_{p_j}=1$ for all $j$, so $$\dim_{\k} H^0(X,\mc{Q}(1))=\sum_{j=1}^{\ell} \dim_{\k} \mc{Q}(1)_{p_j}=\ell,$$ and the sections $t_j$ span $H^0(X,\mc{Q}(1))$. Hence $H^0(\rho\tsr 1)$ is surjective, and the result follows from Lemma \ref{dimension count}.

\end{proof}

\begin{cor}\label{key result}
Assume that $X=\Proj S$ satisfies Assumptions \ref{ass}. Suppose that for each irreducible component $W_i \subset X$, there exists a surjective graded ring homomorphism $\varphi_i:S\to S_i$ such that $W'_i=\Proj S_i$ is normal, and the closed immersion $f_i:W'_i\to X$ induced by $\varphi_i$ has image equal to $W_i$. If $\v=\coprod f_i$ separates singularities and $\dim_{\k} \mc{Q}_p=1$ for all $p\in {\rm Sing}(X)$, then
$$h^0_X(\O_X(1)) = \sum_i h^0_{W'_i}(\O_{W'_i}(1)) - |{\rm Sing}(X)|.$$
\end{cor}
%
%
%

\section{Examples}
\label{examples}

We now present three examples that illustrate the results of the preceding sections. Example 1 is a comprehensive calculation using all of the main theorems of the paper to prove that $\t_A:A\to \mathbf{B}(A)$ is surjective and to compute its kernel. In Example 2, $\t_A$ is shown to be an isomorphism by establishing regular sequences in each $\mathbf{U}_d(A)$ and appealing to  Proposition \ref{kernel-image via local cohomology}. And in Example 3, we define an algebra $A$ for which $\t_A$ is neither injective nor surjective.

Throughout this section, we use $\Proj \mathbf{U}_d(A)$ as a model for $\U_d(A)$; see Proposition \ref{model for Gd}. Also recall from Section \ref{structure} the $\k$-linear isomorphism $\widetilde{\i}: A_d \to {\mathbf U}_d(A)_1$. Via the isomorphism of Proposition \ref{alt pres}, we henceforth view $\widetilde{\i}$ as a map $A_d\to S'_1\subset S/K$ in the notation of Section \ref{Td}. 

In each example, the algebra $A$ is presented by generators and relations as a quotient of the free algebra $\k\la x,y,z\ra$. If $w\in A_d$ is a monomial, meaning it can be represented by a word $w_0w_1\cdots w_{d-1}$ where $w_i\in \{x,y,z\}$, we suppress tensors and write $\widetilde{\i}(w)=w_0w_1\cdots w_{d-1}\in S'_1.$ In Examples 1 and 2, $A$ has a  monomial $\k$-basis of the form $\{x^iy^jz^k\ :\ i, j, k\in \N\}.$ In these cases, it is convenient to set $m_{ijk}=\widetilde{\i}(x^iy^jz^k).$

\subsection{Example 1.}
Let $$A = \k \la x, y, z \ra/\la xy-yx, yz-zy, zx \ra.$$ This is the algebra $S'(0,0)$ in the classification of \cite[Theorem 5.2]{CG1}; it is a twisted tensor product of $\k[x,y]$ and $\k[z]$. In particular, $\{x^i y^j z^k\}$ is a $\k$-linear basis of $A$ and the Hilbert series of $A$ is $(1-t)^{-3}$.  Our goal is to show that the algebra $\mbfB(A)$ is $1$-generated, and to 
prove that $\mbfB(A)\cong A/\la xyz \ra$. 

As discussed above, for any $(i, j, k) \in \N^3$, $i+j+k = d$, let $m_{i, j, k} = \widetilde{\i}(x^iy^jz^k)$. We need the following technical lemma.

\begin{lemma}\label{zero products}
Let $d \geq 2$. Let $\frak{m} = {\mathbf U}_d(A)_{+}$. The following hold:
\begin{itemize}
\item[(1)] if $ijk \geq 1$, then $\frak{m}$ annihilates $m_{i, j, k}$;
\item[(2)] $m_{d,0,0}$ annihilates $m_{0,i,d-i}$ for $i < d$, and $m_{i,0,d-i}$ for $i \leq d-2$;
\item[(3)] $m_{0,d,0}$ annihilates $m_{i,0,d-i}$ for $1 \leq i \leq d-1$;
\item[(4)] $m_{0,0,d}$ annihilates $m_{i,d-i,0}$ for $1 \leq i \leq d-1$, and $m_{i,0,d-i}$ for $i \geq 2$;
\item[(5)] for all $1 \leq i \leq d-1$, $m_{i,0,d-i}$ annihilates $m_{j,d-j,0}$ for $j < d$, $m_{0,j,d-j}$ for $j > 0$, and $m_{j,0,d-j}$ for $j \notin \{i-1, i, i+1\}$;
\item[(6)] if $i \geq 1$ and $j < d$, then $m_{i,d-i,0} m_{0,j,d-j} = 0$.
\end{itemize}
\end{lemma}

\begin{proof}

Let $S=\Sym(A_1)^{\tsr d}$ and let $K$ be the multilinearization of the ideal $\la xy-yx, yz-zy, zx \ra$, as in Definition \ref{multilinearization ideal}. Recall that ${\mathbf U}_d(A)$ is isomorphic to the subalgebra of $S/K$ generated by the $m_{i,j,k}$ by Proposition \ref{alt pres}. In the ring $S/K$, for any index $i$, $0 \leq i \leq d-1$ we have 
\begin{align*}
&x_i y_i z_{i+1} = x_i y_{i+1} z_i = x_{i+1} y_i z_i = 0, \\
&x_i y_{i+1} z_{i+1} = x_{i+1} y_i z_{i+1} = x_{i+1} y_{i+1} z_i = 0.
\end{align*}

For (1), let $(i, j, k) \in \N^3$ and $ijk \geq 1$. Then note that $m_{i, j, k}$ contains $x_{i-1} y_i z_{i+1}$. Let $w' = m_{i', j', k'}$ for some arbitrary $(i', j', k') \in \N^3$, and let $w = w'm_{i,j,k}$. We consider three cases. First, if $i'-1 \geq i$, then $x_i$ occurs in $w'$. So $x_i y_i z_{i+1}$ is in $w$, hence $w = 0$. Second, if $i' \leq i$ and $i'+j' \leq i$, then $z_i$ occurs in $w'$. So $x_{i-1} y_i z_i$ is in $w$, hence $w = 0$. Third, if $i' \leq i$ and $i'+j' \geq i+1$, then $y_i$ is in $w'$. So $x_{i-1} y_i z_{i+1} y_i$ is in $w$. We have $$x_{i-1} y_i z_{i+1} y_i = x_{i-1} y_i z_i y_{i+1} = 0,$$ hence $w = 0$. 

Proofs of statements (2)-(6) are similar and are left to the reader.

\end{proof}

Next we determine the closed points of the scheme $\U_d(A)$.

 \begin{prop}\label{type (2) Ore point scheme}
Let $d \geq 2$. Let $e_1 = [1:0:0]$, $e_3 = [0:0:1]\in \P^2$. The set of closed points of the  scheme $\U_d(A) \subset (\P^2)^{\times d}$ consists of the following:
\begin{align*}
& [x:y:0]^{d} \text{ for all } [x:y] \in \P^1; \\
& [0:y:z]^{d} \text { for all } [y:z] \in \P^1; \\
& e_1^{i} \times [x: 0: z] \times e_3^{d-i-1} \text{ for all } [x:z] \in \P^1 \text{ and } 0 \leq i \leq d-1.
\end{align*}
\end{prop} 

\begin{proof}
The closed points of $\U_d(A)$ are the closed points of $(\P^2)^{\times d}$ that satisfy the equations: $$x_i y_{i+1} - y_i x_{i+1} = 0, \quad y_i z_{i+1} - z_i y_{i+1} = 0, \quad z_i x_{i+1} = 0,$$ 
for all  $0 \leq i \leq d-1.$ The result follows by a straightforward induction.

\end{proof}

Define the following basic affine open subsets of $\U_d(A)$:
$$U_{i} = D_+(m_{i,0,d-i}) \text{ for } 0 \leq i \leq d, \text{ and } U_{d+1} = D_+(m_{0,d,0}).$$

\begin{lemma}\label{type (2) Ore open cover}
The collection $\{U_i\}_{0\leq i \leq d+1}$ is an open cover of $\U_d(A)$. 
Moreover, $U_i \cap U_j$ is nonempty if and only if $j \equiv i-1 \text{ or } j\equiv i+1\pmod{d+2}$.
\end{lemma}

\begin{proof}
The scheme $\U_d(A)$ is quasi-compact, so, for the first statement, it suffices to check that every closed point is in one of the $U_i$. This is easy to see based on Proposition \ref{type (2) Ore point scheme}. For the second statement recall that $D_+(f) \cap D_+(g) = D_+(fg)$ and apply Lemma \ref{zero products}.

\end{proof}

 We now prove that $\U_d(A)$ is reduced and determine its singular points.

\begin{thm}\label{Ore case (2) is reduced}
Let $X = \U_d(A)$ for some $d\ge 2$. Then the following hold.
\begin{itemize}
\item[(1)] For all $0 \leq i \leq d+1$, $\mc{O}_X(U_i) \cong \k[s,t]/(st)$.
\item[(2)] $X$ is a reduced scheme.
\item[(3)] Let $p_i = e_1^{i} \times e_3^{d-i}$ for $0 \leq i \leq d$, and let $p_{d+1} = e_2^d$ in $(\P^2)^{\times d}$. The set of singular points of $X$ is ${\rm Sing}(X) = \{p_i : 0 \leq i \leq d+1\},$ and each singular point is a nodal singularity.
\end{itemize}
\end{thm}

\begin{proof}


Statement (1) follows from Lemma \ref{zero products} and a Gr\"obner basis for the defining ideal of ${\mathbf U}_d(A)$. For example, for $i = 0$, using Lemma \ref{zero products} we have $$\mc{O}_X(U_0) = \k\Big[\frac{m_{0,d,0}}{m_{0,0,d}}, \frac{m_{1,0,d-1}}{m_{0,0,d}}\Big].$$ Moreover, $m_{0,d,0}m_{1,0,d-1} = 0$. Then one uses the Gr\"obner basis to conclude that $$\mc{O}_X(U_0) \cong \k[s, t]/(st).$$ 

For statement (2), the ring $\k[s, t]/(st)$ is reduced since it is a quotient of a polynomial ring by a square-free monomial ideal. Then $X$ is a reduced scheme since the $U_i$ form an open affine cover of $X$.

Let $p \in X$ be any point. It suffices to consider regularity of $p$ on the affine open cover $\{U_i\}_{0 \leq i \leq d+1}$. Suppose that $p \in U_i$. From statement (1), $U_i \cong \Spec \k[s,t]/(st)$. It is then clear that $p$ is regular unless $p = p_i$. The statement that each $p_i$ is nodal is also obvious.

\end{proof}
%
%

Now we determine the irreducible components of the scheme $\U_d(A)$. We will use the notation established in Section \ref{Hilbert series}. The following result can easily be checked using a Gr\"obner basis. We omit the proof.

\begin{lemma}\label{presentation of A/<y>}
For any $d \geq 2$, there is an isomorphism of graded algebras $${\mathbf U}_d(A/\la y \ra) \cong \frac{\k[m_{i, d-i} : 0 \leq i \leq d]}{(m_{i,d-i} m_{j,d-j} : j >i+1)}.$$
\end{lemma}
~\\

Using this result, for each $0 \leq i \leq d-1$, define a algebra map $\psi_i: {\mathbf U}_d(A/\la y \ra) \to \k[s,t]$ by: $$m_{i, d-i} \mapsto s, \quad m_{i+1, d-i-1} \mapsto t, \quad m_{i', j'} \mapsto 0 \text{ otherwise.}$$ Let $\pi_y : A \to A/\la y \ra$ be the quotient map and define $\varphi_i: {\mathbf U}_d(A) \to \k[s,t]$ as the composite ${\mathbf U}_d(\pi_y) \circ \psi_i$. Let $\pi_z: A \to A/\la z \ra$ and $\pi_x: A \to A/\la x \ra$ be quotient maps and define $\varphi_{d-1} = {\mathbf U}_d(\pi_z)$ and $\varphi_d = {\mathbf U}_d(\pi_x)$. Observe that $A/\la z \ra \cong \k[x,y]$ and $A/\la x \ra \cong \k[y,z]$, hence, by Theorem \ref{functor on commutative} we have isomorphisms $${\mathbf U}_d(A/\la z \ra) \cong \k[x,y]_{(d)} \quad \text { and } \quad {\mathbf U}_d(A/\la x \ra) \cong \k[y,z]_{(d)}.$$

Noting that the $\varphi_i$ are surjective, we define the corresponding closed immersions $f_i: W_i' \to \U_d(A)$, where $$W_i' = \begin{cases}\Proj \k[s,t] \quad &0 \leq i \leq d-1, \\ \Proj \k[x,y]_{(d)} \quad &i = d-1, \\
\Proj \k[y,z]_{(d)} \quad &i = d.\end{cases}$$

For $0 \leq i \leq d+1$, let $W_i$ be the scheme-theoretic image of $f_i$. Note that, as schemes, $W'_i \cong \P^1_{\k}$ for all $i$. In particular, $W'_i$ is normal and irreducible. It's clear from the description of the closed points of $\U_d(A)$ in Lemma \ref{type (2) Ore point scheme} and the definition of the $f_i$ that the set of all closed points of $\U_d(A)$ is contained in $\cup_{i} W_i$. Since $\U_d(A)$ is of finite type over $\k$, the set of closed points is dense, so $\U_d(A) = \cup_{i} W_i$. Obviously there are no (nontrivial) containments between these irreducible closed subspaces, so we have proved the following.

\begin{lemma} The irreducible components of $\U_d(A)$ are the $W_i$, $0 \leq i \leq d+1$.
\end{lemma}
Observe that we have proved that the scheme $X = \U_d(A)$ satisfies Assumptions \ref{ass}. Recall that $\v:X' \to X$ is the normalization of $X$, where $(X',\v)=\left(\coprod W'_i, \coprod f_i\right)$ and the structure maps $\psi_i:W'_i\to X'$ satisfy $\v\circ \psi = f_i$. Since $\v = \coprod f_i$, we have $\v_*\O_{X'} = \bigoplus (f_i)_*\O_{W'_i}$. Also recall the short exact sequence
$$0\to \O_X\xto{\v^{\#}} \v_*\O_{X'} \xto{\rho} \mc{Q}\to 0,$$
where $\mc{Q}=\v_*\O_{X'}/\O_X$. The sheaf $\mc{Q}$ is a direct sum of skyscraper sheaves supported on ${\rm Sing}(X)$.

\begin{lemma}\label{stalks of skyscraper}
Let $X = \U_d(A)$. Then:
\begin{itemize}
\item[(1)] $\dim X = 1$;
\item[(2)] $\v=\coprod f_i$ separates singularities;
\item[(3)] $\dim_{\k} \mc{Q}_p=1$ for all $p\in {\rm Sing}(X)$.
\end{itemize}
\end{lemma}

\begin{proof}
Statement (1) is clear since $\dim W_i = \dim \P^1_{\k} = 1$ for all $i$. 

Recall that the set of singular points was determined in Theorem \ref{Ore case (2) is reduced}. One then checks that $|W_i \cap {\rm Sing}(X)| = 2$ for all $i$. Statements (2) and (3) follow from this fact and statement (1).

\end{proof}

Now we can compute the Hilbert series of the ring $\mbfB(A)$.

\begin{thm}\label{Hilbert series of B(A), Ore case 2}
Let $A = \k\la x, y, z \ra/\la xy-yx, yz-zy, zx \ra$. Then the Hilbert series of $\mbfB(A)$ is given by $$H_{\mbfB(A)}(t) = 1 + \sum_{d \geq 1} 3d \, t^d.$$
\end{thm}

\begin{proof}
We've checked that all of the hypotheses of Corollary \ref{key result} hold for the scheme $\U_d(A)$, $d \geq 2$. It is well known  that the dimension of the space of global sections of $\mc{O}_{W'_i}(1)$ is given by $$h^0(\mc{O}_{W'_i}(1)) = \begin{cases} 2 \quad &0 \leq i \leq d-1 \\ d+1 \quad & i = d-1, \, d. \end{cases}$$ By Theorem \ref{Ore case (2) is reduced}, there are $d+2$ singular points. Hence by Corollary \ref{key result}, we have $\dim_{\k} \mbfB(A)_d = 2d+2(d+1) - (d+2) = 3d.$

\end{proof}


%
%

Recall that in Section \ref{structure} we defined $$J_d = \{f \in A_d : Pf = 0 \text{ for every truncated right point module $P$ of length $d+1$}\}.$$ 
Since the scheme $\U_d(A)$ is reduced, Corollary \ref{kernel of tau} and its proof shows that $H^0_{\mf m}({\mathbf U}_d(A))_1 = \widetilde \iota(J_d),$ where ${\mf m} = {\mathbf U}_d(A)_{+}$.

\begin{prop}\label{local cohomology in degree 1}
We have $$H^0_{\mf m}({\mathbf U}_d(A))_1 = \begin{cases} 0 \quad &d = 2 \\
\bigoplus_{i, j, k \geq 1} \k \, m_{ijk} \quad &d \geq 3.\end{cases}$$
\end{prop}

\begin{proof}
Let $P$ be a truncated point module of length $d+1$. Recall, by Proposition \ref{type (2) Ore point scheme}, that $P$ corresponds to a unique closed point of $(\P^2)^{\times d}$ of the form:
\begin{align*}
& (1) \quad [x:y:0]^{d} \text{ for all } [x:y] \in \P^1; \\
& (2) \quad [0:y:z]^{d} \text { for all } [y:z] \in \P^1; \\
& (3) \quad e_1^{i} \times [x: 0: z] \times e_3^{d-i-1} \text{ for all } [x:z] \in \P^1 \text{ and } 0 \leq i \leq d-1.
\end{align*}
We will say that a truncated point module $P$ corresponding to a point from (1), (2), (3) is of type (1), (2), (3), respectively. 

First consider the case $d = 2$. Let $f \in J_2$ and write $f = \sum \l_{i,j,k} \, x^i y^j z^k$ for some $\l_{i,j,k} \in \k$. Since $f$ annihilates any $P$ of type (1) and $\k$ is infinite, $\l_{i,j,k} = 0$ for $k = 0$; then, since $f$ annihilates any $P$ of type (2), $\l_{i, j, k} = $ for $i = 0$. We conclude that $f = 0$, and the result follows.

Now suppose that $d \geq 3$. Note that the sum in the statement is direct because $\widetilde \iota$ is a $\k$-linear isomorphism. In Lemma \ref{zero products}(1) we showed that $\bigoplus_{i, j, k \geq 1} \k \, m_{ijk} \subseteq H^0_{\mf m}({\mathbf U}_d(A))_1$. To prove the reverse containment it suffices to show: $$f \in J_d, \quad f = \sum_{ijk = 0} \l_{i,j,k} \, x^iy^jz^k \quad \implies \quad f = 0.$$
This can be argued identically as in the previous paragraph.
 
\end{proof}

We use the facts above to prove the desired results about ${\mbfB}(A)$.

%

\begin{thm}\label{tau is surjective, Ore case 2}
Let $A = \k \la x, y, z \ra/\la xy-yx, yz-zy, zx \ra.$ Then:
\begin{itemize}
\item[(1)] the algebra morphism $\t_A: A \to \mbfB(A)$ is surjective; \item[(2)] $\mbfB(A)$ is $1$-generated;
\item[(3)] $\ker \, \t_A = \la xyz \ra$.
\end{itemize}
Consequently, $\mbfB(A) \cong A/\la xyz \ra$.
\end{thm}

\begin{proof}
The first two statements follow from Proposition \ref{kernel-image via local cohomology} once we show that $H^1_{\mf m}({\mathbf U}_d(A))_1 = 0$ for all $d \geq 2$.
Recall the exact sequence: $$\xymatrix{
0\to H^0_{\mf m}({\mathbf U}_d(A))_1  \ar[r] & {\mathbf U}_d(A)_1  \ar[r]^-{\z_1} & H^0(\U_d(A), \mcL_d) \ar[r] & H^1_{\mf m}({\mathbf U}_d(A))_1 \to 0.
}$$
Hence we have $$\dim_{\k} H^1_{\mf m}({\mathbf U}_d(A))_1 = \dim_{\k} H^0_{\mf m}({\mathbf U}_d(A))_1 - \dim_{\k} {\mathbf U}_d(A)_1+ \dim_{\k} H^0(\U_d(A), \mcL_d).$$ By Theorem \ref{Hilbert series of B(A), Ore case 2}, $\dim_{\k} H^0(\U_d(A), \mcL_d) = 3d$ for all $d \geq 2$. Proposition \ref{local cohomology in degree 1} implies that $\dim_{\k} H^0_{\mf m}({\mathbf U}_d(A))_1 = \binom{d-1}{2}$ for all $d \geq 3$. And since $\widetilde \iota: A_d \to {\mathbf U}_d(A)_1$ is a $\k$-linear isomorphism, $\dim_{\k} {\mathbf U}_d(A)_1 = \binom{d+2}{2}$ for all $d \geq 2$. It follows that $\dim_{\k} H^1_{\mf m}({\mathbf U}_d(A))_1 = 0$ for all $d \geq 2$.

Note that it follows from Proposition \ref{local cohomology in degree 1} that $(\ker \, \t_A)_d $ is the $\k$-linear span of the $x^iy^jz^k$, $ijk \geq 1$. It is very easy to prove, using the defining relations of $A$, that the space $\la xyz \ra_d$ is also equal to this span. Hence  $\ker \, \t_A = \la xyz \ra$.

\end{proof}
%

\subsection{Example 2.} 
Let $$A=\k\la x,y,z\ra/\la xy-yx,\ zx,\ zy\ra.$$ This is the algebra $S(0,0)$ from \cite[Theorem 5.2]{CG1}. We will use Proposition \ref{kernel-image via local cohomology} to prove that $\mbfB(A)\cong A$ by showing, in Theorem \ref{regSeqOreCase3} below, that ${\rm depth}\ {\mathbf U}_d(A)\ge 2$ for all $d\ge 2$. 

Again, via Proposition \ref{alt pres}, we view each ${\mathbf U}_d(A)$ as the subalgebra of 
$$\frac{\k[x_i, y_i, z_i : 0\le i\le d-1]}{(x_iy_{i+1}-y_ix_{i+1}, z_ix_{i+1}, z_iy_{i+1} : 0\le i\le d-2)}$$
generated by elements of the form $m_{ijk}=\widetilde{\i}(x^iy^jz^k)$ for $i+j+k=d$. Because we will not fix $d$ and instead work with the entire family of algebras ${\mathbf U}_d(A)$, we let $S(d)=\k[x_i, y_i, z_i : 0\le i\le d-1]$ and $$K(d)=(x_iy_{i+1}-y_ix_{i+1}, z_ix_{i+1}, z_iy_{i+1} : 0\le i\le d-2).$$ One readily checks that, in the standard grading, the ideal $K(d)$ has a quadratic Gr\"obner basis with respect to the degree-lexicographic order where $x_i<y_i<z_i<x_{i+1}<y_{i+1}<z_{i+1}$ for all $0\le i\le d-2$.

The algebra $S(d)$ carries many (multi)gradings of which we highlight a few:
\begin{itemize}
\item An $\N^{3d}$-grading, where the components of the multidegree correspond to each of the $3d$ generators. We denote the multidegree component corresponding to $x_i$ by $\deg_{x_i}$ and similarly for $y_i$ and $z_i$.
\item An $\N^d$-grading which collapses the $\N^{3d}$-grading by summing $\deg_{x_i}, \deg_{y_i}, \deg_{z_i}$ for each $i$.
\item $\N$-gradings which collapse the $\N^{3d}$-grading by forgetting all multidegree components except a single $\deg_{z_i}$. We call this the $z_i$-degree.
\item An $\N^3$-grading which collapses the $\N^{3d}$-grading by summing all $x_i$-degree components $\deg_{x_0},$ $\deg_{x_1},\ldots, \deg_{x_{d-1}}$, and similarly for $y$ and $z$. We call the components of this multidegree $x$-, $y$-, and $z$-degrees, denoted $\deg_x$, $\deg_y$, and $\deg_z$. 
\end{itemize}

The generators of $K(d)$ are homogeneous with respect to the two $\N^d$-gradings, the $\N^3$ and the $\N$-grading, so these are inherited by $S(d)/K(d)$. 
Recall from Proposition \ref{alt pres} that $\mathbf{U}_d(A)$ is isomorphic to the subalgebra of $S(d)/K(d)$ generated in $\N^d$-degree $(1,1,\ldots,1)$, and $\mathbf{U}_d(A)$ is thereby endowed with another $\N$-grading, which we call the $U$-degree, in addition to the inherited $z_i$-degree gradings. The $U$-degree of a homogeneous element $q\in \mathbf{U}_d(A)$ will be denoted by $\deg_U q$.

Our depth calculation will be inductive, making use of certain subalgebras of ${\mathbf U}_d(A)$. For each $0\le k\le d$, 
let 
\begin{align*}
D(d,k) &= \k[ m_{i,d-i-k,k} : 0\le i\le d-k]\subseteq {\mathbf U}_d(A),\quad\text{and}\\
R(d,k) &= \prod_{j=0}^k D(d,k)\subseteq {\mathbf U}_d(A)
\end{align*}
be unital subalgebras. Clearly, ${\mathbf U}_d(A) = R(d,d)$. 

Since the generators of the subalgebras $D(d,k)$ and $R(d,k)$ have $U$-degree 1, these subalgebras inherit the $U$-degree $\N$-grading from $\mathbf{U}_d(A)$. Similarly, since the generators of $D(d,k)$ and $R(d,k)$ are homogeneous with respect to $z_i$-degree, these subalgebras are also graded by $z_i$-degree.

Observe that for $1\le k\le d$ there are (non-graded) $\k$-algebra maps 
\begin{align*}
\psi_{d,k}:D(d,k)&\to D(d-k,0)\cong\k[x,y]_{(d-k)} \\ 
m_{i, d-i-k,k} &\mapsto x^i y^{d-i-k}
\end{align*}
and
\begin{align*}
 \xi_{d,k}:D(d,k-1)D(d,k)&\to R(d-k+1,1)\\
 m_{i,d-i-k+1,k-1}&\mapsto m_{i,d-i-k+1,0}\\
 m_{i,d-i-k,k}&\mapsto m_{i,d-i-k,1}
\end{align*}
where $\k[x,y]_{(d-k)}$ is the $(d-k)$-Veronese subalgebra of $\k[x,y]$. To see these maps are well-defined, recall that the only terms containing a $z_i$ in the quadratic Gr\"obner basis for $K(d)$ are $z_ix_{i+1}$ and $z_iy_{i+1}$, and such expressions do not arise in subalgebras generated by $m_{i,d-i-k+1,k-1}$ and $m_{i,d-i-k,k}$.

We record a few useful facts about $D(d,k)$, $\psi_{d,k}$, and $\xi_{d,k}$ for use below. The proofs are straightforward.

\begin{lemma}
\label{example2facts}
For all $d\ge 2$, and all $0\le k<d$,
\begin{enumerate}
\item The map $\psi_{d,k}$ is an isomorphism. In particular, 
\begin{enumerate}
\item $D(d,k)$ is a domain,
\item $D(d,k)\cong D(d-k,0)$,
\item $\{m_{d-k,0,k}, m_{0,d-k,k}\}$ is a regular sequence in $D(d,k)$.
\end{enumerate}
\item The map $\xi_{d,k}$ is an isomorphism.
\item Whenever $|i-j|>1$, we have $D(d,i)_+D(d,j)_+=0$.
\end{enumerate}
\end{lemma}

That ${\mathbf U}_d(A)$ can be written as a product of subalgebras satisfying the orthogonality condition in Lemma \ref{example2facts}(3), together with the fact that $\xi_{d,k}$ is an isomorphism for all $d$ and $k$, suggests the approach of constructing regular sequences in the algebras ${\mathbf U}_d(A)$ by finding regular sequences in the subalgebras $R(d,1)$, and combining them with their images under the maps $\xi_{d,k}^{-1}$. Indeed, this will be our approach.

For each $d\ge 2$ and $0\le k\le d$, let 
$$u_{d,k} = \sum\limits_{i=0}^{k} m_{d-i,0,i}\quad\text{ and }\quad v_{d,k}=\sum\limits_{j=0}^{k} m_{0,d-j,j}.$$
The goal of this example is to prove that $\{u_{d,d}, v_{d,d}\}$ is a regular sequence in $\mathbf{U}_d(A)$. 
We will start by proving that $\{u_{d,1}, v_{d,1}\}$ is regular sequence in $R(d,1)$ for all $d\ge 2$, and then prove inductively that for all $d\ge 2$ and all $1\le k\le d$, $\{u_{d,k}, v_{d,k}\}$ is a regular sequence in $R(d,k)$. The following technical lemma is key to this approach.

\begin{lemma}
\label{syzygyOreCase3}
Let $\varphi:R(d,1)^2 \to R(d,1)$ be given by 
$$\varphi(g,h)=gm_{d,0,0}+hm_{0,d,0}.$$
Then $\ker \varphi$ is generated by $(m_{0,d,0},\ -m_{d,0,0})$ and, for all $0\le j\le d-2,$
$$(m_{0,d-1,1}m_{j,d-j-1,1},\ -m_{d-1,0,1}m_{j+1,d-j-2,1}).$$
Moreover, if $q\in R(d,1)$ is $U$- and $z_{d-1}$-homogeneous such that $\deg_U q > \deg_{z_{d-1}} q$ then for all $0\le j\le d-2,$ $$q(m_{0,d-1,1}m_{j,d-j-1,1},\ -m_{d-1,0,1}m_{j+1,d-j-2,1})\in R(d,1)(m_{0,d,0},\ -m_{d,0,0}).$$
\end{lemma}

\begin{proof} We provide only an outline of the argument, leaving the details to the reader.

Using the Gr\"obner basis for $K(d)$, it is straightforward to check that the listed pairs are contained in the kernel, and that the last statement holds.

To show the list of generators is sufficient, let $(g,h)\in \ker\ \varphi.$ Since $\varphi$ is both $U$- and $z_{d-1}$-homogeneous, we may assume the same of $(g,h)$. The $\N^3$ grading inherited by $R(d,1)$ enables us to assume without loss of generality that  $\deg_y g\ge d$ and $\deg_x h\ge d$.

If $g, h$ have $z_{d-1}$-degree 0, then the fact that $\{m_{d,0,0}, m_{0,d,0}\}$ is a regular sequence in $D(d,0)$ implies $(g,h)\in D(d,0)(m_{0,d,0},\ -m_{d,0,0})$.

For the higher $z_{d-1}$-degree cases (degree 1 and degree $\ge 2$) observe that the relations $x_iy_{i+1}-y_ix_{i+1}\in K(d)$ imply relations of the form
$$m_{a,d-a,0}m_{b,d-b,0}-m_{a+1,d-a-1,0}m_{b-1,d-b+1,0},$$ 
$$m_{a,d-a,0}m_{b,d-b-1,1}-m_{a+1,d-a-1,0}m_{b-1,d-b,1},$$ 
$$m_{a,d-a-1,1}m_{b,d-b-1,1}-m_{a+1,d-a-2,1}m_{b-1,d-b,1}$$
hold in $R(d,1)$. 

If $g,h$ have $z_{d-1}$-degree 1, the fact that $\deg_y g\ge d$ and $\deg_x h\ge d$, together with the first two types of relations above imply that $g\in R(d,1)_+m_{0,d,0}$ and $h\in R(d,1)_+m_{d,0,0}$. A Gr\"obner basis argument shows $m_{d,0,0}m_{0,d,0}$ is regular in $R(d,1)$, so $(g,h)\in D(d,0)(m_{0,d,0},\ -m_{d,0,0})$ in this case as well.

If $g,h$ have $z_{d-1}$-degree $\ge 2$, the argument is analogous, incorporating the third relation type above to show $g\in R(d,1)_+m_{0,d-1,1}$ and $h\in R(d,1)_+m_{d-1,0,1}$. Note that $m_{d-1,0,1}m_{0,d-1,1}$ has neither the required $x$-degree nor the $y$-degree to be a term of $g$ or $h$.

\end{proof}

\begin{lemma}
\label{annihilatorsOreCase3}
Let $d\ge 2$ and for each $0\le k\le d$ let
$$J(d,k)=\begin{cases} (x_i, y_i\ :\ d-k+1\le i\le d-1) & 2\le k\le d\\ (0) & k=0,1\end{cases}$$
be an ideal of $S(d)/K(d)$. Then 
$${\rm ann}_{R(d,k)}(m_{0,d-k,k})={\rm ann}_{R(d,k)}(m_{d-k,0,k})=J(d,k)\cap R(d,k).$$
\end{lemma}

\begin{proof}
 It is obvious that $J(d,k)\cap R(d,k)$ is contained in the indicated annihilators. Since $R(d,0)=\k[x,y]_{(d)}$ is a domain, the result holds when $k=0$. 
For $k>0$, let $$T(d,k) = \k[x_i, y_i, z_k\ :\ 0\le i\le d-1, j\ge d-k-1]\subseteq S(d).$$ Observe that the ring $T'=T(d,k)/(J(d,k)\cap T(d,k))$ is generated by 
$$\{x_i, y_i\ :\ 0\le i\le d-k\} \cup \{z_j\ :\ d-k\le j\le d-1\}.$$
It is clear from the Gr\"obner basis of $K(d)$ that $T'$ is a domain. Thus the algebra $R(d,k)/(J(d,k)\cap R(d,k))$ is a domain. The result follows.

\end{proof}

\begin{lemma}
\label{regEltsOreCase3}
For any $d\ge 2$ and $0\le k\le d$, the elements $u_k$ and $v_k$ are regular in $R(d,k)$.
\end{lemma}

\begin{proof} Fix an arbitrary $d\ge 2$. Recall that $u_{d,0}$ is regular in $R(d,0)=D(d,0)$ by Lemma \ref{example2facts}(1)(iii). 

Let $k\ge 1$ and suppose $a\in R(d,k)$ is $U$-degree homogeneous such that $au_{d,k}=0$. 
Since $R(d,k)$ is also graded by $z_{d-k}$-degree, if $a\neq 0$, there exists a minimal $t\ge 0$ such that 
$a=a_0+a_1+\cdots+a_t$, where $\deg_{z_{d-k}} a_i = i$ and $a_t\neq 0$.

%

In the top $z_{d-k}$-degree, we have $a_tm_{d-k,0,k}=0$. 
We claim that $a_t=0$, which implies $a=0$ by the minimality of $t$. If $k=1$, $a_t=0$ is immediate from Lemma \ref{annihilatorsOreCase3}.
If $k\ge 2$, then by Lemma \ref{annihilatorsOreCase3}, 
$a_t \in J(d,k)\cap R(d,k).$
Since $a_t\in J(d,k)$, there exists $d-k+1\le j\le d-1$ such that $\deg_{x_j} a_t>0$ or $\deg_{y_j} a_t>0$. For such a $j$ we  have $x_j\cdot (z_{d-k}\cdots z_{d-1})=0$ in $S(d)/K(d)$, and similarly for $y_j$. 
Since $a_t\in R(d,k)$, $\deg_{z_{d-i}} a_t\ge t$ for all $1\le i\le k$. It follows that $t=0$. 

Since $t=0$, then $a_0m_{d-k,0,k}=0$ and $au_{d,k}=0$ imply $a_0u_{d,k-1}=0$, which holds in $R(d,k-1)$. By induction, since $u_{d,k-1}$ is regular in $R(d,k-1)$, $a_0=0$.
%
The regularity of $v_k$ is proved analogously.

\end{proof}

\begin{prop}\label{baseRegularSequence}
For all $d\ge 2$, $\{u_{d,1}, v_{d,1}\}$ is a regular sequence in $R(d,1)$. 
\end{prop}

\begin{proof} Fix an arbitrary $d\ge 2$ and for readability let $u_0=u_{d,0}$ and $v_0=v_{d,0}$, $u_1=u_{d,1}$, and $v_1=v_{d,1}$.
Suppose $au_1+bv_1=0$ for $U$-degree homogeneous $a,b\in R(d,1)$. By Lemma \ref{regEltsOreCase3} it suffices to prove that exists a $U$-homogeneous element $q\in R(d,1)$ such that $a=qv_1$ and $b=-qu_1$.

We may decompose $a$ and $b$ as
$$a=a_0+a_1+\cdots+a_t,\qquad b=b_0+b_1+\cdots + b_{t'},$$
for some $t, t' \geq 0$, where the $a_i$ and $b_i$ are $U$-homogeneous, and $\deg_{z_{d-1}} a_i = \deg_{z_{d-1}} b_i = i$. Without loss of generality, we assume $a_t\neq 0$ and $b_{t'}\neq 0$.

Recall that $R(d,1)$ is $z_{d-1}$-graded. If $t<t'$, examining the highest $z_{d-1}$-degree term of $au_1+bv_1=0$ yields $b_{t'}m_{0,d-1,1}=0.$ As in the proof of the Lemma \ref{regEltsOreCase3}, this implies $b_{t'}=0$, a contradiction. Similarly, $a_t=0$ if $t>t'$, so we conclude $t=t'$. We now establish that each $a_i$ and $b_i$ has a particular form.
\bigskip

\noindent {\bf Claim:} Let $q_{-1}=0$. There exist $q_0, q_1, \ldots, q_t\in R(d,1)$ with $\deg_{z_{d-1}} q_i = i$ such that
$a_i = q_{i-1}m_{0,d-1,1} + q_iv_{0}$ and
$b_i = -q_{i-1}m_{d-1,0,1} - q_iu_{0}.$

\begin{proof}[Proof of Claim]
We proceed by induction on the indices of the $q_i$, examining the component of $au_1+bv_1=0$ in $z_{d-1}$-degree $i$. 

The $z_{d-1}$-degree 0 component of $au_1+bv_1=0$ is $a_0u_{0}+b_0v_{0}=0$. This equation holds in $D(d,0)$, where $\{u_0, v_0\}$ is a regular sequence, by Lemma \ref{example2facts}(1)(iii). Thus we obtain $a_0=q_0v_{0}$ and $b_0=-q_0u_{0}$ for some $q_0\in D(d,0)$, establishing the base case of the induction.

Let $i>0$. The $z_{d-1}$-degree $i$ component of $au_1+bv_1=0$  is
$$a_{i-1}m_{d-1,0,1}+b_{i-1}m_{0,d-1,1}+a_iu_{0}+b_iv_{0}=0.$$
By the induction hypothesis,  
\begin{align*}
&q_{i-1}m_{0,d,0}m_{d-1,0,1}+a_im_{d,0,0}-q_{i-1}m_{d,0,0}m_{0,d-1,1}+b_im_{0,d,0}=0.
\end{align*}
We rewrite this as
$$(a_i-q_{i-1}m_{0,d-1,1})m_{d,0,0} + (b_i+q_{i-1}m_{d-1,0,1})m_{0,d,0}=0.$$
By Lemma \ref{syzygyOreCase3}, and the $z_{d-1}$-grading on $R(d,1)$, there exist $q_i, Q_0,\ldots, Q_{d-2}\in R(d,1)$ such that 
$\deg_{z_{d-1}} q_i=i$, $\deg_{z_{d-1}} Q_j=i-2$ for $0\le j\le d-2$, and 
\begin{align*}
a_i&=q_{i-1}m_{0,d-1,1} + q_im_{0,d,0} + m_{0,d-1,1}\sum_{j=0}^{d-2}Q_j m_{j,d-j-1,1}\\ 
b_i&=-q_{i-1}m_{d-1,0,1} - q_im_{d,0,0} - m_{d-1,0,1}\sum_{j=0}^{d-2}Q_{j} m_{j+1,d-j-2,1}.
\end{align*}

By Lemma  \ref{syzygyOreCase3}, there is no loss of generality in assuming $\deg_U Q_j = \deg_{z_{d-1}} Q_j=i-2$ for all $0\le j\le d-2$. 

We argue that, in fact, all of the $Q_j$ are zero. To see this, consider the component of $au_1+bv_1=0$ in $z_{d-1}$-degree $i+1$:
$$a_{i}m_{d-1,0,1}+b_{i}m_{0,d-1,1}+a_{i+1}u_{0}+b_{i+1}v_{0}=0,$$
where, if $i=t$, then we interpret $a_{i+1}=b_{i+1}=0$.
Examining $U$-degrees, we have
\begin{align*}
\deg_U \left(q_im_{0,d,0}m_{d-1,0,1}-q_im_{d,0,0}m_{0,d-1,1}\right)&\ge i+2,\quad\text{and, if $i<t$,}\\
\deg_U \left(a_{i+1}u_{0}+b_{i+1}v_{0}\right) &\ge i+2,
\end{align*}
but the $U$-degree of
$$m_{d-1,0,1}m_{0,d-1,1}\sum_{j=0}^{d-2}Q_j m_{j,d-j-1,1}-m_{0,d-1,1}m_{d-1,0,1}\sum_{j=1}^{d-1}Q_j m_{j,d-j-1,1}$$is $i+1$.
Since $a_i, b_i$ are $U$-homogeneous, the latter expression must vanish. By inspecting multidegrees, we see that $Q_j=0$ for all $j$.

\end{proof}

The Claim implies that for $\displaystyle q=\sum_{i=0}^{t-1} q_i$, we have
$a=qv_1 + q_tv_{0}$ and $b=-qu_1 - q_tu_{0},$
so it suffices to prove $q_tm_{d-1,0,1}=q_tm_{0,d-1,1}=0$, which implies $q_t=0$ by Lemma \ref{annihilatorsOreCase3}. 

To see this, observe, as in the proof of the claim above, that the component of $au_1+bv_1=0$ in $z_{d-1}$-degree $t+1$ is
$$q_t(m_{0,d,0}m_{d-1,0,1}-m_{d,0,0}m_{0,d-1,1})=0.$$
The fact that $m_{0,d,0}m_{d-1,0,1}$ and $m_{d,0,0}m_{0,d-1,1}$ have different multidegrees implies $q_tm_{d-1,0,1}=q_tm_{0,d-1,1}=0$, as desired.

\end{proof}

Finally we are ready to prove our desired result.

\begin{thm} \label{regSeqOreCase3}
For all $d\ge 2$ and all $1\le k\le d$, $\{u_{d,k}, v_{d,k}\}$ is a regular sequence in $R(d,k)$. In particular, $\{u_{d,d}, v_{d,d}\}$ is a regular sequence in $\mathbf{U}_d(A)$.
\end{thm}

\begin{proof}
The result holds for all $d\ge 2$ and $k=1$ by Proposition \ref{baseRegularSequence}.

Now fix $d\ge 2$, suppose that $k>1$, and write $u_k=u_{d,k}$ and $v_k=v_{d,k}$. By Lemma \ref{example2facts}(2), 
we know that $\{u_{d,k}-u_{d,k-2},v_{d,k}-v_{d,k-2}\} $ is a regular sequence in $D(d,k-1)D(d,k)$. 

Suppose $au_{k}+bv_{k}=0$ for $U$-homogeneous elements $a,b\in R(d,k)$. Again, by Lemma \ref{regEltsOreCase3}, it suffices to prove that exists a $U$-homogeneous element $Q\in R(d,k)$ such that $a=Qv_k$ and $b=-Qu_k$.

Write $a=a'+a''+a'''$ where $a'\in R(d,k-2)$, $a''\in R(d,k-2)D(d,k-1)_+=D(d,k-2)D(d,k-1)_+$, and $a'''\in R(d,k-1)D(d,k)_+=D(d,k-1)D(d,k)_+$. Note that these subalgebra equalities follow from Lemma \ref{example2facts}(3). Similarly decompose $b=b'+b''+b'''.$

By induction, $\{u_{k-1},v_{k-1}\}$ is a regular sequence in $R(d,k-1)$, so there exists some $Q_1\in R(d,k-1)$ such that $a'+a''=Q_1v_{k-1}$ and $b'+b''=-Q_1u_{k-1}$. Similarly, since $\{u_{k}-u_{k-2}, v_{k}-v_{k-2}\}$ is a regular sequence in $D(d,k-1)D(d,k)$, there exists $Q_2\in D(d,k-1)D(d,k)$ such that $a''+a'''=Q_2(v_{k}-v_{k-2})$ and $b''+b'''=-Q_2(u_{k}-u_{k-2})$.

We further decompose $Q_1 = Q_1^{nz}+Q_1^{>}+Q_1^{=}$ where $Q_1^{nz}\in R(d,k-2)$, $Q_1^{=}$ consists of terms where $\deg_U = \deg_{z_{d-k+1}}$, and $Q_1^{>}$  consists of terms where $\deg_U > \deg_{z_{d-k+1}}$. Analogously, we decompose $Q_2=Q_2^{nz} + Q_2^{\ge}$ where $Q_2^{nz}\in D(d,k-1)$ and $Q_2^{\ge}\in D(d,k-1)D(d,k)_+$. Observe that 
\begin{align*}
a' &= Q_1^{nz}v_{k-2}\\
a'' &= Q_1^{nz}m_{0,d-k+1,k-1}+(Q_1^{>}+Q_1^{=})v_{k-1}\\
&=Q_2^{nz}m_{0,d-k+1,k-1}\\
a'''&=Q_2^{nz}m_{0,d-k,k}+Q_2^{\ge}(v_{k}-v_{k-2})\\
&=(Q_2^{nz}+Q_1^{>}+Q_1^{nz})(m_{0,d-k,k}) + Q_2^{\ge}v_{k}
\end{align*}
and from the relation involving $a''$, and the fact that $D(d,k-1)$ is a domain, $Q_2^{nz}=Q_1^{=}$. Thus
\begin{align*}
a&= Q_1v_{k-1}+(Q_2^{nz}+Q_1^{>}+Q_1^{nz})(m_{0,d-k,k}) + Q_2^{\ge}v_{k}\\
&=Q_1v_{k-1}+Q_1m_{0,d-k,k} + Q_2^{\ge}v_{k}\\
&=(Q_1+Q_2^{\ge})v_{k}.
\end{align*}
A similar calculation shows $b=-(Q_1+Q_2^{\ge})u_{k}$, completing the proof.

\end{proof}
~\\

\subsection{Example 3.}
~\\

Let $\k$ be a field such that ${\rm char \ } \k \neq 2$, and let $$A=\dfrac{\k\la x,y,z\ra}{\la x^2-xy, yx, zx-xz-z^2, zy-yz\ra}.$$ Note that $z\in A_1$ is normal, and $x\in R_1$ is normal, where $$R=A/\la z \ra = \k \la x,y \ra/ \la x^2-xy,yx \ra.$$ Since $R/\la x\ra=\k[y]$, we see that $A$ is Noetherian by \cite[Lemma 1.11(2)]{rogalski}. 

A straightforward (noncommutative) Gr\"obner basis calculation, with $x < y < z$ and left degree-lexicographic order, shows that
$$\{1, x, x^2, xz, zx, x^2z, xzx\}\cup \bigcup_{n=1}^{\infty} \{y^n, y^{n-1}z\}$$
is a $\k$-basis for $A$, and thus the Hilbert series of $A$ is
$$h_A(t)=1+3t+5t^2+4t^3+\dfrac{2t^4}{1-t}.$$

However, $A$ is not a quotient of a quadratic, 3-dimensional AS-regular algebra. If such an algebra, call it $S$, admitted a surjective algebra map $S\to A$, then identifying the generators of $S$ with those of $A$, and recalling that $S$ must be a domain, implies 
$$x^2-xy+\a yx,\ xz-zx+z^2-\b yx,\quad zy-yz-\g yx$$
is a set of defining quadratic relations for $S$. One can check that these relations force $\dim_{\k} S_3=9$, contradicting the fact that the Hilbert series of a quadratic, 3-dimension AS-regular algebra is $(1-t)^{-3}$.


We will compute the Hilbert series of $\mbfB(A)$, from which it will be apparent that the homomorphism $\displaystyle  \t_A: A\to \mbfB(A)$ of Proposition \ref{map A to B} is neither injective nor surjective. We begin by describing the closed points and an open cover of $\U_d=\U_d(A)$ for each $d\ge 2.$  

\begin{lemma}
\label{points of G}
The closed points of the scheme $\U_2$ are
\begin{align*}
(1:0:0)\times (1:1:0),&\quad (0:0:1)\times (1:0:1)\\
(0:1:0)\times (0:1:0),&\quad (1:0:-1)\times (0:0:1),
\end{align*}
and the sets $U_1=D_+(x_0x_1),\ U_2=D_+(z_0x_1),\ U_3=D_+(y_0y_1),\ U_4=D_+(x_0z_1)$ form a disjoint open cover of $\U_2$.

The closed points of the scheme $\U_3$ are
$$(1:0:-1)\times (0:0:1)\times (1:0:1)\quad\text{and}\quad (0:1:0)\times (0:1:0)\times (0:1:0),$$
and the sets $D_+(x_0z_1x_2)$ and $D_+(y_0y_1y_2)$ form a disjoint open cover of $\U_3$.

For $d\ge 4$, $\U_d$ has one closed point, $(0:1:0)^{\times d}$, and the set  $D_+(y_0y_1\cdots y_{d-1})$ is an open cover of $\U_d$ for $d\ge 4$.
\end{lemma}

\begin{proof}
The closed points of $\U_d$ are the closed points of $(\P^2)^{\times d}$ at which the following forms vanish for all $0\le i\le d-1$:
$$x_ix_{i+1}-x_iy_{i+1},\ y_ix_{i+1},\ z_ix_{i+1}-x_iz_{i+1}-z_iz_{i+1},\ z_iy_{i+1}-y_iz_{i+1}.$$
Verification that the sets described in the Lemma are the solutions to these is left to the reader.

The description of the open covers follows readily, since closed points of open subschemes are closed, for schemes that are locally finite over $\Spec \k$.

\end{proof}


Lemma \ref{points of G} implies that
\begin{align*}
\mbfB(A)_2&={\mathbf U}_2(A)(1)_{(x_0x_1)}\oplus {\mathbf U}_2(A)(1)_{(z_0x_1)}\oplus {\mathbf U}_2(A)(1)_{(y_0y_1)}\oplus {\mathbf U}_2(A)(1)_{(x_0z_1)}\\
\mbfB(A)_3&={\mathbf U}_3(A)(1)_{(x_0z_1x_2)}\oplus {\mathbf U}_3(A)(1)_{(y_0y_1y_2)},\ \text{and}\\
\mbfB(A)_d&={\mathbf U}_d(A)(1)_{(y_0y_1\cdots y_{d-1})},\ \text{for } d\ge 4.
\end{align*}

The dimensions of these local rings can be analyzed via the corresponding localization of the algebra $S/K$, where $S=\Sym(A_1)^{\tsr d}$ and $K$ is the multilinearization of $\la x^2-xy, yx, zx-xz-z^2, zy-yz\ra$ described in Definition \ref{multilinearization ideal}. The following list gives a Gr\"obner basis for the ideal $K$, sorted by multidegree, with respect to degree-reverse-lexicographic order with  $z_{i+1}<y_{i+1}<x_{i+1}<z_i<y_i<x_i$ for $0\le i\le d-2$.

\begin{align*}
\text{For all } d\ge 2&:\\
(1,1)&\quad x_ix_{i+1}-x_iy_{i+1},\ y_ix_{i+1},\ z_ix_{i+1}-x_iz_{i+1}-z_iz_{i+1},\ z_iy_{i+1}-y_iz_{i+1}\\
(1,2)&\quad x_iy_{i+1}z_{i+1}+y_iz_{i+1}^2,\ \\
(2,1)&\quad x_iy_iz_{i+1}+y_iz_iz_{i+1},\ x_i^2z_{i+1}+x_iz_iz_{i+1}+y_iz_iz_{i+1},\ x_iy_iy_{i+1},\ \\
(1,3)&\quad y_iy_{i+1}z_{i+1}^2,\ \\
(2,2)&\quad y_i^2z_{i+1}^2\\
(3,1)&\quad y_i^2z_iz_{i+1},\ \\
\text{If } d\ge 3 \text{ add}&:\\
(1,1,1)&\quad x_iz_{i+1}z_{i+2}+z_iz_{i+1}z_{i+2},\ x_iy_{i+1}y_{i+2},\ y_iz_{i+1}z_{i+2},\\
(1,1,2)&\quad y_iy_{i+1}z_{i+2}^2,\ x_iy_{i+1}z_{i+2}^2,\ \\
(1,2,1)&\quad x_iy_{i+1}^2z_{i+2},\ \\
(2,1,1)&\quad x_i^2y_{i+1}z_{i+2},\ \\
\text{If } d\ge 4 \text{ add}&:\\
(1,1,1,1)&\quad z_iz_{i+1}z_{i+2}z_{i+3},\ 
\end{align*}

\begin{prop}
\label{hs of B}
The Hilbert series of $\mbfB(A)$ is
$$1+3t+6t^2+3t^3+\dfrac{2t^4}{1-t}.$$

\end{prop}

\begin{proof} That $\dim_{\k}\mbfB(A)_0 = 1$ and $\dim_{\k}\mbfB(A)_1 = 3$ is an immediate consequence of the construction of $\mbfB(A)$. In higher degrees, we calculate directly, making repeated use of the Gr\"obner basis described above.

We begin in degree two. A minimal generating set for ${\mathbf U}_2(A)$ is
$$x_0x_1=\widetilde{\i}(x^2),\ x_0z_1=\widetilde{\i}(xz),\ y_0y_1=\widetilde{\i}(y^2),\ y_0z_1=\widetilde{\i}(yz),\ z_0x_1=\widetilde{\i}(zx).$$
Note that the following expressions vanish in ${\mathbf U}_2(A)$:
$$(x_0x_1)^2x_0z_1,\ (x_0x_1)y_0y_1,\ (x_0x_1)y_0z_1,\ (x_0x_1)z_0x_1,$$
$$(z_0x_1)x_0x_1,\ (z_0x_1)^2x_0z_1,\ (z_0x_1)y_0y_1,\ (z_0x_1)y_0z_1.$$
It follows that $\dim_{\k} {\mathbf U}_2(A)(1)_{(x_0x_1)}=\dim_{\k} {\mathbf U}_2(A)(1)_{(z_0x_1)}=1$.
Also note that
$$(y_0y_1)x_0x_1,\ (y_0y_1)x_0z_1,\ (y_0y_1)z_0x_1,\ (y_0z_1)^2$$
vanish in ${\mathbf U}_2(A)$, but $(y_0y_1)^k(y_0z_1)$ is nonzero for all $k\ge 0$. This shows that $\dim_{\k} {\mathbf U}_2(A)(1)_{(y_0y_1)}=2.$
Similarly,
$$(x_0z_1)y_0y_1,\ (x_0z_1)(x_0x_1+y_0z_1),\ (x_0z_1)(x_0x_1+z_0x_1),\ (x_0z_1)(x_0x_1)^2$$
vanish in ${\mathbf U}_2(A)$, but $(x_0z_1)^kx_0x_1\neq 0$ for all $k\ge 0$. Thus
$\dim_{\k} {\mathbf U}_2(A)(1)_{(x_0z_1)}=2$. We conclude $\dim_{\k}\mbfB(A)_2=6$.

Next we consider degree three. A minimal generating set for ${\mathbf U}_3(A)$ is
$$x_0x_1z_2=\widetilde{\i}(x^2z),\ x_0z_1x_2=\widetilde{\i}(xzx),\ y_0y_1y_2=\widetilde{\i}(y^3),\ y_0y_1z_2=\widetilde{\i}(y^2z).$$
The following expressions vanish in ${\mathbf U}_3(A)$:
$$(x_0z_1x_2)x_0x_1z_2,\ (x_0z_1x_2)y_0y_1y_2,\ (x_0z_1x_2)y_0y_1z_2$$ 
$$(y_0y_1y_2)x_0x_1z_2,\ (y_0y_1y_2)x_0z_1x_2,\ (y_0y_1z_2)^2.$$
But $(y_0y_1y_2)^ky_0y_1z_2\neq 0$ for all $k\ge 0$. Thus we have $\dim_{\k} {\mathbf U}_3(A)(1)_{(x_0z_1x_2)}=1$ and  $\dim_{\k} {\mathbf U}_3(A)(1)_{(y_0y_1y_2)}=2$, so $\dim_{\k}\mbfB(A)_3=3$.

Finally, let $d\ge 4$. The elements $m_{0,d,0}=\widetilde{\i}(y^d)$ and $m_{0,d-1,1}=\widetilde{\i}(y^{d-1}z)$ generate ${\mathbf U}_d(A)$. We have $$(y_0y_1\cdots y_{d-2}z_{d-1})^2=0\quad\text{and}\quad (y_0y_1\cdots y_{d-1})^k(y_0y_1\cdots y_{d-2}z_{d-1})\neq 0$$ for all $k\ge 0$. Therefore $\dim_{\k}\mbfB(A)_d=\dim_{\k} {\mathbf U}_d(A)(1)_{(y_0y_1\cdots y_{d-1})}=2$ for all  $d\ge 4.$

\end{proof}

Consideration of the Hilbert series of $A$ given above, together with Proposition \ref{hs of B} establishes that $\t_A:A\to \mbfB(A)$ fails to be surjective in degree 2 and fails to be injective in degree 3. Moreover, it follows from the proof of Proposition \ref{hs of B} that $\U_d$ is not reduced for each $d>2$, since certain affine open subschemes of each $\U_d$ contain nonzero nilpotent elements. The interested reader can check that $\ker\ \t_A=\la zxy\ra = {\rm span}_{\k}\{zxy\}$, and $\mbfB(A)/\la \e\ra\cong A/\la zxy\ra$ where $\e\in \mbfB(A)_2$ is not in the image of $\t_A$. So $\t_A: A \to \mbfB(A)$ is an isomorphism in degrees $\geq 4$. 

The reader desiring a simple example where $\t_A$ is not surjective in all higher degrees may wish to consider $A=\k\la x,y\ra/\la x^2, xy\ra$.

\bibliographystyle{plain}
\bibliography{bibliog2}

\end{document}